\documentclass[11pt, letterpaper]{amsart}

\usepackage[pagebackref,hypertexnames=false, colorlinks, citecolor=red, linkcolor=red]{hyperref}
\usepackage{amssymb}
\usepackage{mathrsfs}

	\normalbaselines						  %

\newcommand{\cz}{Calder\'{o}n--Zygmund\ }

\newcommand{\shto}{\raisebox{.3ex}{$\scriptscriptstyle\rightarrow$}\!}

\newcommand{\R}{{\mathbb  R}}

\newcommand{\D}{{\mathbb  D}}

\newcommand{\T}{\mathbb{T}}

\newcommand{\Z}{{\mathbb  Z}}
\newcommand{\N}{{\mathbb  N}}
\newcommand{\C}{{\mathbb  C}}

\newcommand{\fD}{{\mathfrak  D}}
\newcommand{\OZ}{{\mathbf{O}}}
\newcommand{\dd}{{d}}
\newcommand{\ID}{{\mathbf{1}}}

\newcommand{\OID}{{\mathbf{I}}}

\newcommand{\fdot}{\,\cdot\,}

\newcommand{\wt}{\widetilde}
\newcommand{\bI}{\mathbf{I}}
\newcommand{\1}{\mathbf{1}}

\newcommand{\La}{\langle}
\newcommand{\Ra}{\rangle}

\newcommand{\cH}{\mathcal{H}}
\newcommand{\cK}{\mathcal{K}}
\newcommand{\cD}{\mathcal{D}}

\newcommand{\cM}{\mathcal{M}}
\newcommand{\cV}{\mathcal{V}}
\newcommand{\K}{\mathcal{K}}
\newcommand{\I}{\mathbf{I}}
\newcommand{\cU}{\mathcal{U}}

\newcommand{\te}{\theta}
\newcommand{\f}{\varphi}

\newcommand{\e}{\varepsilon}

\DeclareMathOperator{\clos}{clos}

\DeclareMathOperator{\Ker}{Ker}
\DeclareMathOperator{\Ran}{Ran}
\DeclareMathOperator{\spa}{span}
\DeclareMathOperator{\im}{Im}

\DeclareMathOperator{\supp}{supp}

\DeclareMathOperator{\var}{var}

\newcommand{\la}{\lambda}

\newcommand{\ci}[1]{_{ {}_{\scriptstyle #1}}}
\newcommand{\ti}[1]{_{\scriptstyle \text{\rm #1}}}

\newcommand{\ut}[1]{^{\scriptstyle \text{\rm #1}}}

\catcode`@=11
\chardef\mathlig@atcode\count255

\def\actively#1#2{\begingroup\uccode`\~=`#2\relax\uppercase{\endgroup#1~}}
\def\mathlig@gobble{\afterassignment\mathlig@next@cmd\let\mathlig@next= }
\def\mathlig@delim{\mathlig@delim}
\def\mathlig@defcs#1{\expandafter\def\csname#1\endcsname}
\def\mathlig@let@cs#1#2{\expandafter\let\expandafter#1\csname#2\endcsname}
\def\mathlig@appendcs#1#2{\expandafter\edef\csname#1\endcsname{\csname#1\endcsname#2}}

\def\mathlig#1#2{\mathlig@checklig#1\mathlig@end\mathlig@defcs{mathlig@back@#1}{#2}\ignorespaces}

\def\mathlig@checklig#1#2\mathlig@end{%
 \expandafter\ifx\csname mathlig@forw@#1\endcsname\relax
 \expandafter\mathchardef\csname mathlig@back@#1\endcsname=\mathcode`#1%
 \mathcode`#1"8000\actively\def#1{\csname mathlig@look@#1\endcsname}%
 \mathlig@dolig#1\mathlig@delim
\fi
\mathlig@checksuffix#1#2\mathlig@end
}

\def\mathlig@checksuffix#1#2\mathlig@end{%
\ifx\mathlig@delim#2\mathlig@delim\relax\else\mathlig@checksuffix@{#1}#2\mathlig@end\fi
}
\def\mathlig@checksuffix@#1#2#3\mathlig@end{%
\expandafter\ifx\csname mathlig@forw@#1#2\endcsname\relax\mathlig@dosuffix{#1}{#2}\fi
\mathlig@checksuffix{#1#2}#3\mathlig@end
}

\def\mathlig@dosuffix#1#2{%
\mathlig@appendcs{mathlig@toks@#1}{#2}%
\mathlig@dolig{#1}{#2}\mathlig@delim
}

\def\mathlig@dolig#1#2\mathlig@delim{%
 \mathlig@defcs{mathlig@look@#1#2}{%
 \mathlig@let@cs\mathlig@next{mathlig@forw@#1#2}\futurelet\mathlig@next@tok\mathlig@next}%
 \mathlig@defcs{mathlig@forw@#1#2}{%
  \mathlig@let@cs\mathlig@next{mathlig@back@#1#2}%
  \mathlig@let@cs\checker{mathlig@chck@#1#2}%
  \mathlig@let@cs\mathligtoks{mathlig@toks@#1#2}%
  \expandafter\ifx\expandafter\mathlig@delim\mathligtoks\mathlig@delim\relax\else
  \expandafter\checker\mathligtoks\mathlig@delim\fi
  \mathlig@next
 }%
 \mathlig@defcs{mathlig@toks@#1#2}{}%
 \mathlig@defcs{mathlig@chck@#1#2}##1##2\mathlig@delim{%
  \ifx\mathlig@next@tok##1%
   \mathlig@let@cs\mathlig@next@cmd{mathlig@look@#1#2##1}\let\mathlig@next\mathlig@gobble
  \fi
  \ifx\mathlig@delim##2\mathlig@delim\relax\else
   \csname mathlig@chck@#1#2\endcsname##2\mathlig@delim
  \fi
 }%
 \ifx\mathlig@delim#2\mathlig@delim\else
  \mathlig@defcs{mathlig@back@#1#2}{\csname mathlig@back@#1\endcsname #2}%
 \fi
}%

\catcode`@\mathlig@atcode

\mathchardef\ordinarycolon\mathcode`\:
\def\vcentcolon{\mathrel{\mathop\ordinarycolon}}
\mathlig{:=}{\vcentcolon=}
\mathlig{::=}{\vcentcolon\vcentcolon=}

\numberwithin{equation}{section}

\theoremstyle{plain}
\newtheorem{theo}{Theorem}[section]

\newtheorem{lem}[theo]{Lemma}

\theoremstyle{definition}
\newtheorem{defn}[theo]{Definition}

\theoremstyle{remark}
\newtheorem*{ex*}{Example}
\theoremstyle{remark}
\newtheorem*{exs*}{Examples}
\theoremstyle{remark}
\newtheorem*{rem*}{Remark}
\newtheorem{rem}[theo]{Remark}
\newtheorem*{rems*}{Remarks}

\title[Singular integrals, rank one perturbations and Clark model]{Singular integrals, rank one perturbations and Clark model in general situation}
\author{Constanze~Liaw}
\address{CASPER and Department of Mathematics, Baylor University, One Bear Place \#97328,      
 Waco, TX  76798, USA}
\email{Constanze$\underline{\,\,\,}$Liaw@baylor.edu}
\urladdr{http://sites.baylor.edu/constanze$\underline{\,\,\,}$liaw/}

\author{Sergei~Treil}
 \thanks{Work of S.~Treil is supported by the National Science Foundation under the grant DMS-1301579.}
\address{Department of Mathematics, Brown University, 151 Thayer
Str./Box 1917,      
Providence, RI  02912, USA }
\email{treil@math.brown.edu}
\urladdr{http://www.math.brown.edu/\~{}treil}


\begin{document}

\begin{abstract}
We start with considering rank one self-adjoint perturbations $A_\alpha = A+\alpha(\fdot,\varphi)\varphi$ with cyclic vector $\varphi\in \mathcal{H}$ on a separable Hilbert space $\mathcal H$.
The spectral representation of the perturbed operator $A_\alpha$ is realized by a (unitary) operator of a special type: the Hilbert transform in the two-weight setting, the weights being spectral measures of the operators $A$ and $A_\alpha$. 

Similar results will be presented for unitary rank one perturbations of unitary operators, leading to singular integral operators on the circle.  

This  motivates the  study of abstract singular integral operators, in particular the regularization of such operator in very general settings.  


Further, starting with contractive rank one perturbations we present the Clark theory for arbitrary spectral measures (i.e.~for arbitrary, possibly not inner characteristic functions). We present a description of the Clark operator and its adjoint in the general settings.  Singular integral operators, in particular the so-called normalized Cauchy transform again plays a prominent role.


Finally, we present a possible way to construct the Clark theory for dissipative rank one perturbations of self-adjoint operators. 

These lecture notes give an account of the mini-course delivered by the authors, which was centered around \cite{LiawTreil2009, Regularizations2013, LT}. Unpublished results are restricted to the last part of this manuscript.
\end{abstract}

\dedicatory{To the memory of Cora Sadosky}

\maketitle

\setcounter{tocdepth}{1}
\tableofcontents

\setcounter{tocdepth}{2}

\section{Introduction}

Rank one perturbations play an important role in operator theory and mathematical physics. One of the principal attractions of rank one perturbations is that for such operators almost everything can be explicitly computed, and then the advanced technique of Harmonic Analysis, like the study of fine properties of Cauchy type integrals, or advanced theory of singular integral operators can be applied.  

\subsection{Main players}
\subsubsection{Rank one perturbations}
Self-adjoint rank one perturbations occurred naturally in mathematical physics \cite{weyl}. For example, a change in the boundary condition of a limit-point half-line Schr\"odinger operator from Dirichlet to Neumann, or to mixed conditions, can be reformulated as adding a rank one perturbation (see for example \cite{SIMREV}).


The technique of  rank one perturbations was used in some results on orthogonal polynomials and Jacobi matrices, and there are some interesting applications to free probability  (see e.g.~\cite{AnshelevichAAM, AnshelevichJFA}).
 They also turned out to be useful in the investigation of certain random Hamiltonian systems called Anderson models and the longstanding Anderson localization conjecture \cite{And1958}. Many specializations of this conjecture were studied in literature and the field is still very active (see e.g.~\cite{AizMol1993, denseGdelta, FrSp, Germ, Klo2}, also see \cite{Kir95, Ban09} for a recent account of parts of the field). Rank one perturbations play a role in \cite{AbaLiawPolt, 3D, 2D, SIMWOL}. Recent studies of closely related unitary Anderson models as well as accessible explanations of the physical relevance of these models can be found, e.g.~in \cite{HamzaJoyeStolz2006, HamzaJoyeStolz2009, Joye2012, Stoiciu2007}. The additive perturbation is replaced by a multiplicative one and the dynamical localization behavior is known to be quite similar to its self-adjoint analogue.


\subsubsection{Singular integral operators}
Singular integrals is a classical and actively studied field in Harmonic Analysis, and rank one perturbations serves as a source of very interesting problems. Many results for rank one perturbations are obtained by investigating fine properties of singular integrals. 
For example,  investigation of the boundary behavior of the Cauchy transform of measures lead via the so-called Aronszajn--Krein formula, see \eqref{Ar-Kr}, to  the famous Aronszajn--Donoghue theorem  stating that singular parts of the spectral measures of the family of rank one perturbation by a cyclic vector are mutually singular.  

As for a different example, basic facts about Cauchy transform of a measure coupled with the Aronszajn--Krein formula  \eqref{Ar-Kr} give a proof of the famous Kato--Rosenblum theorem about preservation of the absolutely continuous spectrum for rank one (and automatically for finite rank) perturbations. While the proof for the trace class perturbations using the technique of wave operator is probably more elegant, the approach of singular integrals gives some helpful insights. 

A deep relation between singular integral operators and rank one perturbations is based on the fact that  a unitary  operator realizing the spectral representation of a rank one perturbation is given by a singular integral operator acting $L^2(\mu)\to L^2(\mu_\alpha)$, where $\mu_\alpha$ is the spectral measure of the perturbed operator, see \cite{LiawTreil2009}; we explain this connection in the present notes. We should mention here that the spectral measures $\mu$ and $\mu_\alpha$ can be extremely bad, without any reasonable smoothness, so the above operator gives a natural example of a two weight estimates for Cauchy type operators with extremely ``pathological'' measures.

\subsubsection{Clark measures and Clark model}
In the paper \cite{Clark} that started what is now called the ``Clark theory'' D.~Clark considered all unitary rank one perturbations of  a special contractive operator (the so-called \emph{model operator} with scalar inner characteristic function). He also described the spectral measures and the spectral representations of the perturbed unitary operators.

The spectral measures of these unitary rank one perturbations were later  called the \emph{Clark measures}. Note, that if we fix one such rank one unitary perturbation, then the other unitary rank perturbations are the rank one perturbations of the fixed one. In the original paper \cite{Clark} all the spectral measures were purely singular, but very often the term \emph{Clark measures} (or Clark family of measures) was used for spectral measures of  unitary rank one perturbations of a unitary operator, or for the spectral measures of self-adjoint rank one perturbations of a self-adjoint operator.  

Later many deep function-theoretic results about Clark measures  were proved by A.~Aleksandrov, (see \cite{Aleksandrov1} through \cite{Aleksandrov5}, or see \cite{poltsara2006} for a survey), so sometimes people refer to \emph{Aleksandrov--Clark theory}, or \emph{Aleksandrov--Clark measures}. Extremely significant contributions to the theory were made then by A.~Poltoratskii, who, in particular, proved the a.e.~existence  with respect to the singular part of the measure of the non-tangential boundary values of the so-called \emph{normalized Cauchy transform}, see \cite{NONTAN}. 

We also mention an important book \cite{SAR} by D.~Sarason where many aspects of Clark theory were treated from the point of view of function space theory. In particular, a description of the Clark operator was obtained in the case when the characteristic function $\theta$ is an extreme point of the unit ball in the Hardy space $H^\infty$. 

Within classical analysis many fruitful connections of Clark measures with holomorphic composition operators, rigid functions and the Nehari interpolation problem have been discovered and studied, see for example, \cite{poltsara2006}. Some problems in the theory of Hardy spaces, and more generally of other spaces of analytic functions are closely related to Clark theory. Thus, 
recently M.~Jury \cite{Jury} computed the asymptotic symbols of a certain class of weakly asymptotic Toeplitz operators in terms of the Aleksandrov--Clark measures which occur in the context of rank one perturbations.

\subsection{Plan of the notes} These lecture notes give an account of the mini-course delivered by the authors, which was centered around \cite{LiawTreil2009, Regularizations2013, LT}. Unpublished results are restricted to the last part of this manuscript.

\subsubsection{Self-adjoint and unitary rank one perturbations}
We begin Section \ref{s-SA} with an introduction of self-adjoint rank one perturbations. We then find a unitary operator $V_\alpha$ giving the spectral representation of the perturbed operator, see Theorem \ref{t-repr-V} below;  this operator looks like  a singular integral operator with Cauchy type kernel  $(s-t)^{-1}$, although the formula of the operator looks quite different from the classical singular integral operators of Cauchy type. 

In particular,  the so-called \emph{Rigidity Theorem}, see Theorem \ref{rigTM} below,  holds for such operators: it essentially says that if the formula from Theorem \ref{t-repr-V} gives a bounded operator with trivial kernel, then, after probably a renormalization (multiplication by a non-vanishing weight) of the measure in the target space, we get exactly the unitary operator from the perturbation theory, given by Theorem \ref{t-repr-V}. 

We then give a different representation of the operator $V_\alpha$ that looks more in line with the traditional formulas for singular integral operators. Regularizations of singular kernels, treated later in Section \ref{s-SIO}, play an important role in getting this alternative representation. 

We then present similar results for the unitary rank one perturbations of unitary operators. Everything works out similarly to the self-adjoint case; some formulas for the unitary case might not look as transparent as the ones in the self-adjoint case, but in the unitary case we avoid technical difficulties related to dealing with unbounded operators. 

\subsubsection{Regularizations of singular integral operators}
Section \ref{s-SIO}  is devoted to the theory of regularization of singular kernels, which we believe have applications far beyond the perturbation theory. We show that under  very general assumptions about a singular kernel, its so-called \emph{restricted boundedness} implies the uniform boundedness of all ``reasonable'' regularizations of the corresponding formal singular integral operator. 

The restricted boundedness of the kernel is the weakest boundedness property of the corresponding singular integral operator.  
Usually, it is assumed in the theory of singular integral operators that a singular kernel $K$ blows up on the diagonal $x=y$, so the formal integral representation $Tf(x) = \int K(x,y) f(y) d\mu(y)$ is not well defined. 

However, even if we only start out with a kernel $K$ (without assuming the we are given an operator) for bounded functions $f$ and $g$ with separated compact supports the expression 
\[
\La Tf, g\Ra = \int K(x,y) f(y) g(x) d\mu(y) d\nu(x)
\]
is well defined, and if the ``correct'' estimate $| \La Tf, g\Ra|\le C\|f\|\ci{L^p(\mu)} \|g\|\ci{L^{p'}(\nu)}$, $1/p+1/p'=1$ holds for all such pairs, we say that $K$ is $L^p(\mu)\to L^p(\nu)$ restrictedly bounded. And we show in Section \ref{s-SIO} that if the measures $\mu$ and $\nu$ do not have common atoms and the kernel $K$ is restrictedly bounded, then for any ``reasonable'' regularization $K_\e$ of the kernel the corresponding regularized operators $T_\e$ are uniformly (in $\e$) bounded. This result gives us a way to define for each restrictedly bounded kernel a corresponding singular integral operator.

%

\subsubsection{Clark model for  contractive perturbations of unitary operators}
Section \ref{s-UNITARY} is devoted to the Clark theory in full generality. We start with
unitary rank one perturbations of a unitary operator $U$ by a $*$-cyclic vector. All such perturbations can be parametrized by a scalar parameter $\gamma\in\T$; if one takes $\gamma\in\D$ the resulting operator will be a completely non-unitary (c.n.u.) contraction with defect indices $1$-$1$. For such a contraction a so-called \emph{functional model}, cf.~\cite{SzNF2010} can be constructed; in fact  functional models are the canonical way of investigating non-normal contractions. 

Thus, the perturbed operator $U_\gamma$, $\gamma\in\D$ has two unitarily equivalent representations: one in the spectral representation of $U$ and the other one in the model space for the functional model. The Clark operator  is a unitary operator intertwining these representations. In Clark's original paper \cite{Clark} this operator was constructed for the case of the operator $U$ having purely singular spectrum. In \cite{Clark} the starting point was a c.n.u. contraction with inner characteristic function, which --- after translation to our language --- means that the unitary operator $U$ (and thus all its rank one unitary perturbations $U_\gamma$) has a purely singular spectral measure. 

In the general case (general spectral measure, or equivalently, a general scalar characteristic function) our approach of starting with perturbations of unitary operators looks more natural; in particular, it allowed us to describe the Clark operator. Of course, now when we know all the formulas, it is possible to go in the opposite direction and start with a c.n.u. contraction; but using this approach without knowing the formulas in advance we would have a hard time getting the results. It could well be just our personal preference, but deducing the formulas in our setup starting from a unitary operator was a natural and a straightforward process.

The main problem with the general case of Clark theory is that for general scalar characteristic function the model is vector-valued, i.e.~the model space consists of  vector-valued functions (with values in $\C^2$). Earlier approaches based on function spaces theory, see for example \cite{SAR}, dealt with spaces of scalar-valued functions. Some parts of the Clark operator were obtained using such methods, but for the full operator one had to honestly write down a complete model space and do all the computations.

The adjoint of the Clark operator is described using singular integral operators of Cauchy type. The so-called \emph{normalized Cauchy transform} investigated by A.~Poltor\-atskii, see \cite{NONTAN}, plays a prominent role there.  The Clark operator itself then can be represented via boundary values of the analytic functions. 

In the model theory we adapt the point of view of \emph{coordinate-free model} by N.~Nikolski and V.~Vasyunin, cf.~\cite{Nik-Vas_model_MSRI_1998, Nik-Vas_FunctModels_1989}, where by picking different spectral representations of the minimal unitary dilation one gets different \emph{transcriptions} of the model. We present a ``universal'' representation formula, valid in any transcription, as well as formulas adapted to two popular transcriptions, the Sz.-Nagy--Foia\c{s} transcription and the de Branges--Rovnyak one. 

\subsubsection{Clark model for dissipative perturbations of self-adjoint operators}
The last part, Section \ref{s-SAmodel} is devoted to the Clark model for the dissipative perturbations of a self-adjoint operator. We adapt a common approach that the model space for a dissipative operator is the model space of its Cayley transform (which is a contraction), with one detail: since the original operator lives in $L^2(\R, \mu)$, we, using the standard conformal map between the upper half-plane $\C_+$ and the unit disc $\D$, move the model space to the real line (half-plane). The results in this section were not presented before. 

Note, that the formulas in this section do not look as elegant as in the case of perturbations of unitary operators. Probably, a different approach to the model of dissipative operators would be more appropriate, but we do not know a serious contender yet. As a pure speculation, the de Branges spaces $\mathfrak L(\f)$ could serve as appropriate model spaces for the dissipative perturbations. These spaces were introduced in the first chapter of \cite{dBr-HSAF_1968}, but were not much investigated, unlike the spaces $\mathcal H(E)$ which were investigated in details in \cite{dBr-HSAF_1968} and were subject of extensive research by many authors.

\section{Self-adjoint and unitary rank one perturbations}\label{s-SA}

\subsection{Self-adjoint rank one perturbations}
\label{s:SA-RK1}
For a self-adjoint (possibly unbounded)  operator  $A$   on a separable Hilbert space $\mathcal H$ let us  consider the family of rank-one perturbations $A_\alpha$, $\alpha\in \R$, 
given by
\begin{align}\label{d-SA}
A_\alpha:= A+\alpha(\,\cdot\,,\f)\ci{\cH}\f\qquad\text{on }\cH.
\end{align}
Here, if the operator $A$ is bounded, then $\f$ is a vector in $\mathcal H$. For unbounded $A$, we can consider the wider class of ``singular form-bounded" perturbations where we assume $\varphi\in\mathcal H_{-1}(A)\supset \mathcal H$, where $\mathcal H_{-r}(A)$, $r\in \N$, is the completion of ${\cH}$ with respect to the norm $\| \cdot\|\ci{\cH_{-r}(A)}$,  $\| f\|\ci{\cH_{-r}(A)} = \| (I +|A|)^{-r/2}f\|\ci{\cH}$. In particular, the perturbation $\alpha (\fdot, \f)\f$ can be unbounded (see \cite{KL, LiawTreil2009} and the references within for further details). If $\cH =L^2(\R,\mu)$ and $A=M_t$ is the multiplication by the independent variable, 
\[
M_t f(t) =tf(t), \qquad \forall t\in\R, 
\]
then $\cH_{-1}(A)$ is exactly the collection of measurable functions such that 
\[
\int_\R \frac{|f(t)|^2}{1+|t|} d\mu(t) <\infty. 
\]

For $r\ge2$ the formal expression \eqref{d-SA} does not uniquely determine a self-adjoint operator: For fixed $\alpha$ there is a family of self-adjoint operators corresponding to \eqref{d-SA}. For this reason we do not consider this case, but rather assume that $r<2$.

Without loss of generality we can assume that $\f$ is cyclic for $A$, that is,  $$\mathcal H = \clos\spa\{(A-\lambda \OID)^{-1}\f: \lambda \in \mathbb{C}\setminus \mathbb{R}\}.$$ Otherwise, i.e.~if $\widetilde{\mathcal H} = \clos\spa\{(A-\lambda \OID)^{-1}\f: \lambda \in \mathbb{C}\setminus \mathbb{R}\}\subsetneq \cH$, then we restrict our attention to the action on $\widetilde{\mathcal H}$ as the perturbation is trivial (does nothing) on $\cH\ominus \widetilde{\mathcal H} $. 

Then according to the Spectral Theorem the operator $A$ is unitarily equivalent to the multiplication $M_t$ by the independent variable in a space $L^2(\mu)$ where $\mu$ is a spectral measure of the operator $A$. Spectral measure is of course  not unique, multiplying a spectral measure by a non-vanishing weight (i.e.~by a function $w\in L^1\ti{loc}(\mu)$, $w>0$ $\mu$-a.e.) we get a different spectral measure. 

It is customary in the operator theory and mathematical physics to consider the \emph{canonical} spectral measure to be the spectral measure associated with the ``vector'' $\f$, i.e.~the unique measure $\mu$ such that 
\[
F(\la) := ((A-\la \bI)^{-1} \f, \f) = \int_\R \frac{d\mu(x)}{x-\la}. 
\]
In this case $\f$ is represented by the function $\1$, and the assumption that $\f\in \cH_{-1}(A)$ means simply that $\int_\R (1+|x|)^{-1} d\mu(x) <1$. 

Knowing function $F$ one can say a lot about the spectral measure $\mu$: since the imaginary part of $F$ is (up to the factor $\pi$) the Poisson integral of $\mu$ we can immediately conclude that the density of the absolutely continuous part of $\mu$ is given by the non-tangential boundary values of $\pi^{-1}\im F(z)$ (such values exist a.e.~by classical results). It is also not hard to show that the singular part of $\mu$ is supported on a set where the (non-tangential) boundary values of $\im F$ are infinite. 

In the heart of the theory of rank one perturbations lies the simple fact that there is a simple relation between the function $F$ and the corresponding functions $F_\alpha$ for the perturbed operators. 

Namely,  the following simple formula for the inverse of the rank one perturbation of the identity is well known
\begin{align}
\label{pert_I}
\Bigl(\bI - (\fdot, b) a \Bigr)^{-1} = \bI + \frac{1}{d}\, (\fdot, b) a;  
\end{align}
here $d=1-(b,a)$ is the so-called perturbation determinant, and the operator is invertible if and only if $d\ne0$. The proof of this formula is an easy exercise, we leave it to the reader. 

Using the above formula \eqref{pert_I} one can easily compute the resolvent of the perturbed operator $A_\alpha$, 
\begin{align}\label{singres}
(A_\alpha-\lambda\OID)^{-1}f
&=(A-\lambda\OID)^{-1}f-\frac{\alpha\left((A-\lambda\OID)^{-1}f,\f\right)}{1+\alpha\left((A-\lambda\OID)^{-1}\f,\f\right)}(A-\lambda\OID)^{-1}\f
\end{align}
which immediately implies the relation between the function $F$ and the corresponding functions $F_\alpha$, $F_\alpha(\la):= ( (A_\alpha -\la\bI)^{-1}\f,\f)$, commonly known as the Aronszajn--Krein formula:
\begin{align}
\label{Ar-Kr}
F_\alpha = \frac{F}{1+\alpha F}\,.
\end{align}
If $\mu_\alpha$ denotes the spectral measure of the perturbed operator $A_\alpha$ associated with $\f$, then 
\[
F_\alpha(\la) = \int_\R\frac{d\mu_\alpha(x)}{x-\la}  .
\]
(Note that it is not hard to show that if $\f$ is cyclic for $A$, then $\f$ is also cyclic for $A_\alpha$ and therefore $A_\alpha$ is unitarily equivalent to the multiplication operator $M_t$ in $L^2(\mu_\alpha)$).

Many classical results in perturbation theory can be obtained from the Aronszajn--Krein formula \eqref{Ar-Kr} and classical results about boundary values of the Cauchy transform. 

For example, it is not hard to show that all the absolutely continuous parts of the measures $\mu_\alpha$ are equivalent (i.e.~mutually absolutely continuous), which is just the Kato--Rosenblum theorem for rank one perturbations. Also, the analysis of the singular parts of the measures $\mu_\alpha$ yields the famous Aronszajn--Donoghue theorem, stating that the singular parts of $\mu_\alpha$ are mutually singular.

\subsection{Rank one perturbations and singular integral operators}\label{ss-RSIO}
We find a sufficient condition on the absence of singular spectrum by studying the spectral representation, which comes in the form of a two weight Hilbert transform. Part of this material can be understood as a first example for Section \ref{s-SIO}.

Consider a family of rank one perturbations  given by $A_\alpha := A+\alpha(\cdot,\f)\f$, see \eqref{d-SA}, where $\f\in \cH_{-1}(A)$ is cyclic for $A$.  Let $\mu$ denote the spectral measure of operator $A$ with respect to $\f$, so $A$ is unitarily equivalent to the multiplication operator $M_t$ in $L^2(\mu)$. Let us consider the operator $A$ in its spectral representation, i.e.~let us assume that $A$ \emph{is} the multiplication operator $M_t$ in $L^2(\mu)$. As we discussed before, the assumption that $\f\in\cH_{-1}(A)$ means simply that $\int_\R(1+|x|)^{-1} d\mu(x)<\infty$,

Then the operator $A_\alpha = A+\alpha(\,\cdot\,,\f)\f$ is defined by
\[
A_\alpha = A+\alpha(\cdot,\f)\f= M_t+\alpha(\fdot, \ID)\ci{L^2(\mu)}\ID
\]
on $L^2(\mu)$. 
On the other hand, the operator $A_\alpha$ is unitarily equivalent to the multiplication $M_s$ by the independent variable $s$ in $L^2(\mu_\alpha)$ (we use a different letter for the independent variable here to distinguish between the multiplication operators in $L^2(\mu)$ and $L^2(\mu_\alpha)$). 

We want to find  a unitary operator giving the spectral representation of the operator $A_\alpha$, i.e.~a unitary operator 
\[
V_\alpha:L^2(\mu)\to L^2(\mu_\alpha)
\]
such that 
\[
V_\alpha A_\alpha = M_s V_\alpha. 
\]
We also want $\f$ to be represented by $\1$ in both representations, which translates to additional condition $V_\alpha\f=\1$. 


Theorem 2.1 of \cite{LiawTreil2009} gives the representation of $V_\alpha$ as the Hilbert transform type singular integral. 
\begin{theo}[Representation Theorem]\label{t-repr-V}
Under the above assumptions the spectral representation $V_\alpha: L^2(\mu)\to L^2(\mu_\alpha)$ of $A_\alpha$ is given by
\begin{equation}
\label{repr-V}
V_\alpha f(s) = f(s) -\alpha \int\frac{f(s)-f(t)}{s-t}\,d\mu(t)
\end{equation}
for all compactly supported $C^1$ functions $f$.
\end{theo} 

Without going into the details of the proof, we indicate the proof strategy for bounded operators $A$. The intertwining condition 
\begin{align}
\label{comm-rel}
M_s V_\alpha = V_\alpha A_\alpha = V_\alpha (M_t+\alpha(\fdot,\f)\f)
\end{align}
can be rewritten as 
\[
V_\alpha M_t = M_s V_\alpha - \alpha (\fdot, \f)V_\alpha\f. 
\]

Using induction we get 
\[
V_\alpha M_t^n = M_s^nV_\alpha -\alpha \sum_{k=0}^{n-1} (\fdot, M_t^k \f) M_s^{n-k-1} V_\alpha \f. 
\]
Recalling that $\f\equiv \1$, $V_\alpha \f\equiv \1$, we get by applying the above identity to $\f$ and denoting $f_n(t):= t^n$ we get 
\[
V_\alpha f_n (s)  = s^n -\alpha  \int\ci\R \left(  \sum_{k=0}^{n-1} t^k s^{n-k-1} \right) \,d\mu(t) .
\]
Summing the geometric progression under the integral we get the representation formula \eqref{repr-V} for $f=f_n$, $f_n(t)= t^n$. Linearity of \eqref{repr-V} implies that it holds for all polynomials, and rather standard approximation reasoning allows to extend this formula to the case of compactly supported $C^1$ functions.

This reasoning, of course, works only for bounded operators $A$ (i.e.~when the measure $\mu$ is compactly supported). In the case of unbounded operators the resolvent identity \eqref{singres} is used instead of \eqref{comm-rel}, see \cite{LiawTreil2009} for the details. 


Aside we mention that integral operators represented by formula \eqref{repr-V} are very interesting objects, probably deserving more careful investigation. Without proof we mention one property (see Theorem 2.2 of \cite{LiawTreil2009}), which can be understood as a converse to the latter Representation Theorem.

\begin{theo}[Rigidity Theorem]
\label{rigTM}
Let measure $\mu$ on $\R$ be supported on at least two distinct points and satisfy $\int(1+|t|)^{-1}\,d\mu(t) <\infty$. 
Let $V$ be defined on compactly supported $C^1$ functions $f$ by formula \eqref{repr-V}.

Assume $V$ extends to a bounded operator from $L^2(\mu)$ to $L^2(\nu)$. Assume $\Ker V=\{0\}$.

Then there exists a function $h$ such that $1/h\in L^\infty(\nu)$, and $M_h V$ is a unitary operator from $L^2(\mu)\to L^2(\nu)$ (equivalently, that $V:L^2(d\mu)\to L^2( |h|^{2}\,d\nu)$ is unitary). 
 
Moreover, the unitary operator $U:=M_hV$ gives the spectral representation of the operator $A_\alpha := M_t + \alpha (\fdot, \f)\f$, $\f\equiv \ID$, in $L^2(\mu)$, namely $UA_\alpha = M_s U$, where $M_s$ is the multiplication by the independent variable $s$ in $L^2(\nu)$. 
\end{theo}

The integral in the representation formula looks like a singular integral operator, but not exactly in the traditional sense. The attempt to understand the precise connection with the theory of classical singular integral operators lead us to the theory of regularizations. We describe these results in more detail in Section \ref{s-SIO} on general abstract singular integral operators.

But now, let us first notice that 
\begin{equation*}
(V_\alpha f,g)\ci{L^2(\mu_\alpha )}=-\alpha\iint \frac{f(t)\overline{g(s)}}{s-t}\,\dd\mu(t)\,\dd\mu_\alpha (s)
\end{equation*}
for all $f\in L^2(\mu)$ and $g\in L^2(\mu_\alpha )$ with separated compact supports. This equality is trivial for compactly supported $C^1$ function $f$  and $g$ (with separated compact supports) and can be extended to the general case by a standard approximation argument.

Therefore, the kernel $K(s, t) = 1/(s-t)$ is what we call \emph{restrictedly bounded} kernel, see Definition \ref{d:RestrBd} below. 
%


An application of Theorem \ref{t:reg-bd-01} and Remark \ref{remark} shows the following result.

\begin{theo}\label{t-UNIF} For the measures $\mu$, $\mu_\alpha$ as above, the operators $T_\e: L^2 (\mu)\to L^2(\mu_\alpha)$, 
$$
T_\e f (s): = \int_\R \frac{f(t)}{s-t + i\e} d\mu(t),
$$
are uniformly (in $\e$) bounded.
\end{theo}

Uniform boundedness of the operators $T_\e$ implies that there exists a w.o.t.~limit point of $T_\e$, as $\e\to 0$. In fact, it can be shown that this limit point is unique if $\e\to0^+$ or $\e\to0^-$, so we can say that there exist a w.o.t.-limits $T_{\pm}=\text{w.o.t.-}\lim_{\e\to0^{\pm}} T_\e$. 

The existence of w.o.t.~limits follows, for example, from the lemma below
the fact that for $\im z>0$ and for $\im z<0$ the non-tangential boundary values of $R f\mu (z):= \int_\R \frac{fd\mu(t)}{t-z}$ exist $\mu_\alpha$-a.e. 
\begin{lem}
\label{l:nontan_f.mu}
For any $f\in L^2(\mu)$ the non-tangential boundary values of $Rf\mu (z)= \int_\R \frac{fd\mu(t)}{t-z}$,   $z\in\C_+$ or $z\in\C_-$ exist $\mu_\alpha$-a.e.
\end{lem}

\begin{proof}
The a.e.~convergence with respect to Lebesgue measure (and so with respect to the absolutely continuous part of $\mu_\alpha$) follows from classical facts about boundary values of functions from Hardy spaces: for $f\ge 0$ the function $R f\mu(z)$ has non-positive imaginary part, so composing it with a conformal mapping from the lower  half-plane $\C_-$ we get a bounded analytic function, which has non-tangential limits on $\R$ a.e.~with respect to Lebesgue measure. Representing arbitrary complex-valued function as linear combination of 4 non-negative ones we get the a.e.~existence (with respect to Lebesgue measure) in general case.

To prove the convergence with respect to the singular part $(\mu_\alpha)\ti{s}$ of $\mu_\alpha$ we get by applying functional $\f$ to the  resolvent formula \eqref{singres}
and denoting $f_\alpha =V_\alpha f$ that 
\[
Rf_\alpha\mu_\alpha = \frac{Rf\mu}{1+\alpha R\mu}\,.
\]
By Polotratkii's theorem, see \cite[Theorem 2.7]{NONTAN}, the non-tangential boundary values of $Rf_\alpha\mu_\alpha/R\mu_\alpha$ exist (and coincide with $f_\alpha$) $(\mu_\alpha)\ti s$-a.e. Combining the above identity with the Aronszajn--Krein formula \eqref{Ar-Kr} we get 
\begin{align}
\label{Rf_alpha-Rf}
\frac{Rf_\alpha\mu_\alpha }{R\mu_\alpha } = \frac{Rf\mu}{R\mu} . 
\end{align}
But it follows from the Aronszajn--Krein formula \eqref{Ar-Kr} that the non-tangential boundary values of $F=R\mu$ exist (and equal to $-1/\alpha$). $(\mu_\alpha)\ti s$-a.e. Indeed
\[
\im F_\alpha = \frac{\im F}{|1+\alpha F|^2}
\]
and $\im F_\alpha$ function is the Poisson extension (up to the factor $\pi$) of the measure $\mu_\alpha$. 
Therefore, since the singular part of the measure $\mu_\alpha$ is supported on the subset of $\R$ where non-tangential boundary values of $\im F_\alpha$ equal $+\infty$, we can conclude that the non tangential boundary values of $F$ equal $-1/\alpha$ $(\mu_\alpha)\ti s$-a.e. 

Since the non-tangential boundary values in \eqref{Rf_alpha-Rf} exist $(\mu_\alpha)\ti s$-a.e., we conclude that same for $Rf\mu$.
\end{proof}

%
%
%
%
%
%
%

The above Lemma \ref{l:nontan_f.mu} implies the w.o.t.~convergence of $T_\e$ as $\e\to 0^+$ or $\e\to0^-$. 
Indeed, Lemma \ref{l:nontan_f.mu} implies  the $\mu_\alpha$-a.e. convergence, which, in turn implies that any weakly convergent subsequence of $T_\e f$ converges to the same function (the a.e.~limit). And this, as one can easily see, means that $T_\e f$ has a weak limit as $\e\to 0^+$ or $\e\to0^-$. 

So, we can define the operators $T_\pm$ either as w.o.t.~limits of $T_\e$ as $\e\to 0^\pm$ or define $T_\pm f$ as the non-tangential boundary values of $R f\mu(z)$, $z\in \C_\pm$. 


%

Using the operators $T_\pm$ we obtain an alternative representation formula, 
see Theorem 3.2 of \cite{LiawTreil2009}:
\begin{theo}\label{reg-2}
Let $\mu$ and $\mu_\alpha$ be the spectral measures of $A$ and $A_\alpha$, and let $T_{\pm}$ be as defined above. 

Then $V_\alpha$ can be written as 
\begin{align}
\label{Tweak}
V_\alpha f(s)=f(s)(\ID-\alpha \,T_\pm\ID)+\alpha \,T_\pm f, \qquad \forall f\in L^2(\mu). 
\end{align}
%
\end{theo}

\begin{proof}
Consider operators $V_\alpha^\e:L^2(\mu)\to L^2(\mu_\alpha)$, 
\[
V_\alpha^\e f(s) =  f(s) -\alpha\int\frac{f(s)-f(t)}{s-t+i\e}\,d\mu(t) = f(s)(1- \alpha T_\e \1(s)) + \alpha T_\e f(s)
\]
and notice that for compactly supported $C^1$ function $f$ the functions $V_\alpha^\e f(s)$ converge uniformly and in $L^2(\mu_\alpha)$ to $V_f$ as $\e\to0$. Together with uniform bounds on $T_\e$ this immediately implies that $V_\alpha^\e$ converges in the strong operator topology to $V_\alpha$. 

Taking w.o.t.~limits we arrive  to the representations \eqref{Tweak}. 
\end{proof}

\begin{rem*}
Note that for the existence of the w.o.t.~limits of $T_\e$ it is sufficient to have $\mu_\alpha$-a.e.~convergence on a dense set. As we just discussed  above, for compactly supported $f\in C^1$ the functions $V_\alpha^\e f$ converge uniformly to $V_\alpha f$. It was also shown in the proof of Lemma 
\ref{l:nontan_f.mu} that $T_\e \1(s) = - F(s+i\e)$ converges $\mu_\alpha$-a.e., which immediately implies $\mu_\alpha$-a.e.~convergence of $T_\e f$ for compactly supported $C^1$ functions. 

This approach was used in \cite{LiawTreil2009}. 
\end{rem*}

\subsection{Unitary 
rank one  perturbations}
\label{s:UnitRk1}
In this section we present the analogues of the Representation Theorem \ref{t-repr-V} and the Rigidity Theorem \ref{rigTM} for the case of unitary rank one  perturbations of unitary operators, that were proved in \cite[Section 8]{LT}.

We should mention, that these results cannot be obtained just by taking the Cayley transform of the self-adjoint case, we will explain this in Section \ref{s-SAmodel}. 

In the contrast with the self-adjoint case the description of all unitary rank one perturbations of a unitary operator is not immediately self-evident, but with a little effort one could see that all unitary rank one perturbations of a unitary operator $U$ can be parametrized as 
\begin{align}
\label{U_gamma-SA}
U_{b,\alpha} = U + (\alpha-1) (\fdot, U^*b)b\qquad b\in\cH,\ \|b\|=1, \quad \alpha\in \T. 
\end{align}
The fact that this formula indeed gives us the parametrization of the unitary rank one perturbations can be easily seen in the case $U=\bI$; the general case then is obtained by right multiplying the formula for the perturbation of $\bI$ by $U$. 

In what follows we assume that the vector $b$ is fixed and use the notation $U_\alpha$ for $U_{b,\alpha}$, so our perturbations will be parametrized	 by the scalar parameter $\alpha\in \T:=\{z\in\C:|z|=1\}$.

Since the action of perturbation $(\fdot, U^*b)b$ is trivial (zero) on $(\spa\{U^n b:n\in\Z\})^\perp$, we can ignore what is going on there and assume without loss of generality that $b$ is $*$-cyclic vector for $U$, meaning that $\spa\{U^n b:n\in\Z\}=\cH$. 

Then by the Spectral Theorem $U$ is unitarily equivalent to the multiplication by the independent variable
$\xi$ in $L^2(\mu)= L^2(\T,\mu)$, where $\mu$ is a spectral measure of $U$.  As in the self-adjoint case we fix a spectral measure $\mu$ to be the spectral measure corresponding to the vector $b$, so $\mu$ is a probability measure and the vector $b$ in the spectral representation is given by the function $\1$. 

So, as before let us assume that $U$ is not just unitarily equivalent, but \emph{is} a multiplication operator $M_\xi$ by the independent variable $\xi$ in $L^2(\mu)=L^2(\T, \mu)$, $\mu(\T)=1$ and the rank unitary perturbations $U_\alpha$ are given by \eqref{U_gamma-SA} with $b= \1$. 

It is not hard to show that if $b$ is $*$-cyclic for $U$ then it is also $*$-cyclic for $U_\alpha=U_{b,\alpha}$, so $U_\alpha$ is unitarily equivalent to the multiplication $M_z$ by the independent variable $z$ in $L^2(\mu_\alpha)$. We take for $\mu_\alpha$ the spectral measure corresponding to the vector $b$, so $b=\1$ in the spectral representation of $U_\alpha$ in $L^2(\mu_\alpha)$.

%


Under these assumptions we want to describe the unitary operator giving the unitary equivalence between $U_\alpha$ and its spectral representation of, i.e.~the unitary operator $\cV_\alpha: L^2(\mu)\to L^2(\mu_\alpha)$ such that $\cV_\alpha \1= \1$ and 
%
\begin{align}
\label{ComRel-02}
\cV_\alpha U_\alpha   = M_z \cV_\alpha .
\end{align}

In Theorem 8.1 of \cite{LT} we proved:

\begin{theo}[Representation Theorem]\label{t-repr-V-unitary}
Let $\cV_\alpha: L^2(\mu)\to L^2(\mu_\alpha)$ be a unitary operator satisfying \eqref{ComRel-02} and such that $\cV_\alpha \ID =\ID$ (which means that $\mu_\alpha$ is the spectral measure of $U_\alpha$ corresponding to the cyclic vector $b$, $b(\xi)\equiv\ID$). Then 
\begin{equation}
\label{repr-V-unitary}
\cV_\alpha f (z)= f(z) +(1-\alpha)\int_\T \frac{f(\xi)-f(z)}{1-\bar\xi z}\,d\mu(\xi)\qquad
\text{for all }f\in C^1(\T).
\end{equation}
\end{theo} 
\begin{proof}
The proof goes similarly to the proof of the self-adjoint case (Theorem \ref{t-repr-V} above) for the bounded perturbations sketched above. Namely, using ``linear algebra'' notation, i.e.~identifying $b\in\cH$ with the operator $b:\C\to\cH$, $b(\alpha)= \alpha b$ and denoting by $b^*$  its  adjoint $b^*\cH\to\C$, $b^*:(x) = (x, b)\ci{\cH}$ we can write 
\[
U_\alpha =U + (\alpha-1) b  b_1^\ast= M_\xi + (\alpha-1) b b_1^\ast,  
\]
where $b_1:=U^*b$. 
Then 
the intertwining relationship \eqref{ComRel-02} gives us
\begin{align}
\label{com-rel-03}
\cV_\alpha U = M_z \cV_\alpha +(1-\alpha) (\cV_\alpha b) b_1^\ast.
\end{align}
Inductively one can show that for $n\ge 0$
\[
\cV_\alpha U^n = M_z^n \cV_\alpha +(1-\alpha) \sum_{k=1}^n M_z^{k-1} (\cV_\alpha b)  \left((U^*)^{n-k} b_1\right)^\ast.
\]
Applying this formula to the function $b\equiv \1\in L^2(\mu)$ and recalling that $(U^n b)(\xi)=\xi^n$, $\cV_\alpha b=\1$, $b_1(\xi)\equiv\xi$, $(U_1^*)^{n-k} b_1\equiv \xi^{n-k+1}$ we obtain summing the geometric series
\begin{align}
\label{repr-form-powers-01}
(\cV_\alpha \xi^n )(z)= z^n +(1-\alpha)\int_\T \frac{\xi^n-z^n}{1-\bar\xi z}\,d\mu(\xi).
\end{align}
The action of $\cV_\alpha$ on $\bar\xi^n$, $n\ge0$ is proved similarly.
Namely, taking the adjoint of the intertwining formula $\cV_\alpha U_\alpha =M_z \cV_\alpha$ and right and left multiplying by $\cV_\alpha$ we get that  $\cV_\alpha U_\alpha^\ast = M_{\bar z} \cV_\alpha$, so
\[
\cV_\alpha U^* = M_{\overline z} \cV_\alpha + (1-\overline\alpha)(\cV_\alpha b_1) b^* .
\]
But that is exactly the intertwining relationship \eqref{com-rel-03} with $U^*=U^{-1}$ instead of $U$ and $M_{\overline z} = M_{z^{-1}}$ instead of $M_z$. So applying the same reasoning as above we get that \eqref{repr-form-powers-01} holds also for $n\le 0$, and therefore for all trigonometric polynomials. 

A standard approximation argument concludes the proof. 
\end{proof}


A converse of the Representation Theorem is also true in the unitary setting. Under mild conditions bounded injective operators $\cV:L^2(\mu)\to L^2(\nu)$ that are given by \eqref{repr-V-unitary} induce a Clark family. More precisely, we quote Theorem 8.4 of \cite{LT}.

\begin{theo}[Rigidity Theorem]
Let a probability measure $\mu$ on $\T$ be supported on at least two distinct points.
Let $\alpha\in \T\setminus \{1\}$, and let $\cV f$ be defined for $C^1$ functions $f$ by the right hand side of  \eqref{repr-V}.

Assume $\cV$ extends to a bounded operator from $L^2(\mu)$ to $L^2(\nu)$ and assume $\Ker \cV=\{0\}$.

Then there exists a function $h$ such that $1/h\in L^\infty(\nu)$, and $M_h \cV$ is a unitary operator from $L^2(\mu)\to L^2(\nu)$ (equivalently, that $\cV:L^2(d\mu)\to L^2( |h|^{2}\,d\nu)$ is unitary). 

Moreover, the measure $|h|^2 \nu$ is exactly the Clark measure $\mu_\alpha$ defined as above, and $\cV$ treated as the operator $L^2(\mu)\to L^2(\mu_\alpha)$ is exactly the operator $\cV_\alpha$ from Theorem \ref{t-repr-V-unitary}. 
%
\end{theo}

As in the self-adjoint setting, the Representation Theorem reminds us of singular integral operators. 
Acting as in the self-adjoint case we  show that the kernel $K(z,\xi)=1(1-\overline \xi z)$ is restrictedly bounded (see Definition \ref{d:RestrBd} below). Again, Theorem \ref{t:reg-bd-01} and Remark \ref{remark} show the uniform boundedness of the regularization of the singular integral operator.

\begin{theo}
For the Clark measures $\mu$ and $\mu_\alpha$, the operators $T_r:L^2(\mu)\to L^2(\mu_\alpha)$ given by
\[
T_r f(z) :=\int_\T \frac{f(\xi)d\mu(\xi)}{1- r\overline\xi z}
\]
are uniformly (in $r\in\R_+\setminus\{1\}$) bounded.
\end{theo}

An analog of Lemma \ref{l:nontan_f.mu} holds for the unit circle with essentially the same proof (for a different proof, see \cite[Proposition 8.2]{LT}), so the limits $\lim_{r\to 1^{\mp}} T_r f(z)$ exist $\mu_\alpha$-a.e.~on $\T$.
So we can define operators $T_{\pm}$ as the $\mu_\alpha$-a.e. limits
\begin{align*}
T_\pm f(z) &:=\lim_{r\to 1^\mp} T_r f (z), \qquad z\in\T, 
\intertext{or, equivalently, as w.o.t.~limits}
T_\pm f &:=\text{w.o.t.-}\lim_{r\to 1^\mp} T_r f .
\end{align*}

Replacing the kernel in  \eqref{repr-V-unitary} by $1/(1-r\overline\xi z)$ and  taking the limit as $r\to 1^\mp$,  
we get  an alternative formula for  $\cV_\alpha$.

\begin{theo}
Let $\mu$ and $\mu_\alpha$ be the spectral measures of $U$ and $U_\alpha$ respectively, and let $T_\pm=\text{w.o.t.-}\lim_{r\to 1^\mp} T_r $ (the existence of the limit was just discussed). Then $\cV_\alpha$ has the alternative representation 
\begin{align*}
 \cV_\alpha f  &= [\ID - (1-\alpha)T_\pm\ID] f +(1-\alpha)T_\pm f
 \qquad\forall f\in L^2(\mu).
\end{align*}
\end{theo}

%

\subsection{How unstable can the singular spectrum become?}\label{s-EXA}
By the Kato-Rosenblum theorem we know that the absolutely continuous spectrum remains invariant  under rank one perturbations. But under a rank one perturbation by a cyclic vector, the singular perturbation can change type, as was shown by an example by Donoghue. So the question becomes: To which extent may the spectral properties of the measures $\mu_\alpha$ vary as we change $\alpha$? Much work has been done and many interesting examples were discovered, several are included in \cite{SIMREV}. 

First of all notice that in the context of rank one perturbations for pure point and the singular continuous spectrum can behave quite different. For example, it is possible for $A_\alpha$ to have purely singular continuous spectrum on the interval $[0,1]$ for all $\alpha$. But the same behavior is not possible for pure point spectrum. In fact, the perturbations $A_\alpha$ are pure point for all $\alpha$ if and only if the spectrum is countable.

Another question concerns the type of parameter sets that allow dense singular embedded (in absolutely continuous) spectrum. For several years, all examples exhibited dense singular embedded spectrum only for a Lebesgue measure zero set of parameters $\alpha$.
It came as a surprise when Del Rio, Fuentes and Poltoratskii \cite{riofuepolt2002} proved the existence of a family of rank one perturbations with dense absolutely continuous spectrum and dense singular spectrum for almost every parameter $\alpha$ in an arbitrary (previously given) set $B\subset\R$ and with purely absolutely continuous spectrum for almost every $\alpha \in\R\backslash B$. Their proof uses Clark theory. Via a complicated construction they show the existence of a characteristic function for which the corresponding family of rank one unitary perturbations has the desired properties. In fact, it is possible to produce most any type of singular spectrum in this setting, see \cite{riofuepolt2}. In the latter reference, the following open problem is formulated: Fix an interval $I\subset \R$ and subset $B\subset \R$. Can one find a family of measures $\mu_\beta$ so that $(\mu_\beta)\ti{s}(J)>0$ if and only if $\beta \in B$ and $(\mu_\beta)\ti{ac}(J)>0$ for all $\beta \in \R$ and for every subset $J\subset I$?

A class of examples is concerned with the question of how unstable the spectral type may be, if we do not have absolutely continuous part. A result of Del Rio, Makarov and Simon \cite{denseGdelta} which was independently proved by Gordon \cite{Gordon1997} states the following. Consider $I\subset\supp\mu$ closed and not a singleton. If $\mu_\alpha|_I$ is singular, then the set of $\alpha$'s for which $\mu_\alpha$ is purely singular continuous is a dense $G_\delta$ set.

A converse to this result was presented by C.~Sundberg \cite{sundberg}: For any closed subinterval $I$ which is not a singleton and any $G_\delta$ subset of $\R$, there exists a family of measures corresponding to a family of rank one perturbations such that $\supp \mu\subset I$, $\mu_\alpha$ is purely singular continuous for $\alpha \in G$ and $\mu_\alpha$ is pure point for $\alpha \in \R\backslash G$. In the proof, Sundberg applies Clark theory. He constructs the characteristic function by defining a function on a Riemann surface $\mathcal R$ over the disk $\D$, and then applies the projection from $\mathcal R$ to $\D$.

\subsection{Behavior of the singular continuous spectrum}\label{ss-SC}
To this day, a characterization of the singular continuous part of the perturbed operator's spectral measure in terms of the unperturbed operator remains an open problem.
Several sufficient conditions for the absence of singular continuous spectrum are known (see, for example, \cite{cimaross, SIMREV}). Within the realm of our methods, an application of Theorem \ref{t-UNIF} empowers us with control over singular spectrum of the perturbed operator.

\begin{lem}[Lemma 4.4 of \cite{LiawTreil2009}]
Operators $A_\alpha$, $\alpha\in \R\setminus \{0\}$, have a pure absolutely continuous spectrum on a closed interval $I$, if
\begin{align*}
\int_0^\e x^{-2} w^\ast_I dx =\infty.
\end{align*}
Here $d\mu = w dx+ d\mu\ti{s}$ ($w\in L^1(dx)$) is the Lebesgue decomposition, and $w^\ast_I$ denotes the increasing rearrangement of $w$ on $I$.
\end{lem}
This result allows a construction of unperturbed operators $A$ with arbitrary embedded singular spectrum and for which all of the perturbed operators $A_\alpha$, $\alpha\neq 0$ have no embedded singular spectrum.

The main ingredient of the proof is a well known weak type result on the growth of operator $T$.
With
$
K \tau (s): = \int_\R \frac{ d\tau(t)}{s-t + i\e}
$
 for a Borel measure $\tau$ with $
\int_\R \frac{ |d\tau(t)|}{t^2 + 1}<\infty$:
If $I\subset \R$ is a bounded closed interval such that $\tau\ti{s}|\ci{I}\ne 0$, then there exists a $C>0$ such that $|\{|K\tau|>t\}\cap I|\ge C/t$ for large $t$ -- leads to the following sufficient condition for the absence of singular spectrum for the perturbed operators $A_\alpha$.

\section{Singular integral operators}\label{s-SIO}
\subsection{Preliminaries}
The Hilbert transform $T$
\[
Tf(x) =\int_\R\frac{f(y)dy}{x-y}
\]
is an example of what is usually called a \emph{singular integral operator}. ``Singular'' here means that the kernel $K(x,y)$ of the operator is not integrable in $y$ near the diagonal, so in the formal expression $Tf(x) = \int K(x,y)f(y) dy$ the integral is not well defined. 

In the case of Hilbert transform it is very easy to show that the integral in the sense of principal value is well defined for $C^1$ compactly supported functions, so the operator is defined on a dense set in $L^2$ (and $L^p$, $1<p<\infty$). It also can be shown that it  can be extended to a bounded operator there. 

Moreover, it can be shown that the integral in the sense of principal value exists a.e.~for all $f\in L^p$, $1<p<\infty$; the proof is not as easy as for the $C^1$ functions, and is, in fact, quite involved.  

A part of the operator $V_\alpha$ from Theorem \ref{t-repr-V} looks like the Hilbert transform, with the difference that the integration there is with respect to a general Radon measure $\mu$. And what makes things even more complicated, is that the target space is $L^2(\mu_\alpha)$ with $\mu_\alpha$ being a new measure. 

In the theory of singular integral operators, there are several ways to define such an operator rigorously. One of the accepted ways, is what one would call the \emph{axiomatic} approach.  Namely, to define a singular integer operator $T:L^p(\mu)\to L^p(\nu)$ with kernel $K$ we assume that we are given its bilinear form, defined on a dense subset of $L^p(\mu)\to L^{p'}(\nu)$, $1/p+1/p'=1$. The fact that $T$ is an integral operator with kernel $K$ means simply that 
\begin{align}
\label{ker-SIO}
\La Tf, g\Ra_\nu = \int K(x,y) f(y) g(x)d\mu(y) d\nu(x)
\end{align}
for all (say bounded) $f$ and $g$ with separated compact supports. Since the kernel $K$ blows up only on the diagonal $x=y$, the integral above is well defined. Note, that according to this definition the multiplication operator $M_\f$, $M_\f f =\f f$ is an operator with kernel $K(x,y)\equiv 0$. 

Moreover, it can be shown that any bounded singular integral operator with kernel $K\equiv 0$, where kernel is understood in the sense of \eqref{ker-SIO}, is a multiplication operator. So, according to the axiomatic approach, any two bounded singular integral operators differ by a multiplication operator. 

Another way to define the singular integral operator with kernel $K$ is to consider the truncated operators $T_\e$, 
\[
T_\e f(x) = \int_{|x-y|>\e} K(x,y) f(y)dy
\]
which under usual assumptions about kernel $K$ are well defined for bounded functions $f$ with compact support. And we say that the integral operator with kernel $K$ is bounded if all operators $T_\e$ are uniformly bounded. If the operators $T_\e$ are uniformly bounded, we can take w.o.t.~limit of $T_\e$ as $\e\to 0^+$, so in this case $K$ is indeed a kernel of a bounded singular integral operator in the sense of the axiomatic approach. 

Moreover, in all known examples if an axiomatically defined operator $T$ is uniformly bounded then the operators $T_\e$ are uniformly bounded. And as it turns out, this is not a coincidence, but a corollary of a very general fact. 

\subsection{Setup}
In this paper we assume that $\mu$ and $\nu$ are Radon measures on $\R^d$ and that $K$ belongs to $L^2\ti{loc}(\mu\times\nu)$ off the diagonal $x=y$, meaning that for any $x_0\ne y_0$ there exists a neighborhood $G$ of $(x_0, y_0)\in \R^d\times \R^d$ such that $K\1\ci{G} \in L^2(\mu\times\nu)$. Note, that these assumptions are weaker than what is usually assumed about the kernels of singular integral operators. 

The main results are also true for (at least some) locally compact abelian groups, in particular for tori $\T^d$. Also, since everything is local, the results can be modified to hold on smooth manifolds. 

\begin{defn}
\label{d:RestrBd}
Let $K\in L^2\ti{loc}(\mu\times\nu)$ off the diagonal $x=y$. We say that $K$ is $L^p(\mu)\to L^p(\nu)$ restrictedly bounded if for all $f\in L^\infty(\mu)$, $g\in L^\infty(\nu)$ with separated compact supports
\begin{align}
\label{restr-bdd}
\left| \int K(x,y) f(y) g(x) d\mu(y) d\nu(x)\right| \le C \|f\|\ci{L^p(\mu)}\|g\|\ci{L^{p'}(\nu)}. 
\end{align}
The best constant $C$ in \eqref{restr-bdd} is called the $L^p(\mu)\to L^p(\nu)$ restricted bound of $K$, and denoted by $[K]\ci{L^p(\mu)\to L^p(\nu)}\ut{r}$.  

If the exponent $p$ and the measures $\mu$, $\nu$ are fixed, we will skip $L^p(\mu)\to L^p(\nu)$ and simply say \emph{restrictedly bounded}. 
\end{defn}

Going back, we can see that the operator $V_\alpha$ from Theorem \ref{t-repr-V} is a singular integral operator (in the sense of axiomatic approach) with kernel $K(s,t)=\alpha/(s-t)$. Since $V_\alpha$ is a unitary operator $L^2(\mu)\to L^2(\mu_\alpha)$ its norm is $1$ and therefore the kernel $\alpha/(s-t)$ is restrictedly bounded with the $L^2(\mu)\to L^2(\mu_\alpha)$ restricted norm at most $1$. Equivalently, one can say that the $L^2(\mu)\to L^2(\mu_\alpha)$ restricted norm of the kernel $1/(s-t)$ is at most $1/|\alpha|$.  

Similarly, the operator $\cV_\alpha$ from Theorem \ref{t-repr-V-unitary} is a singular integral operator with kernel  $K(z, \xi)= (1-\alpha)/(1-\overline \xi z)$, $z, \xi \in\T$, and the $L^2(\mu)\to L^2(\mu_\alpha)$ restricted norm of the kernel $1/(1-\overline \xi z)$ is at most $1/|1-\alpha|$. 

\subsection{Regularizations of singular kernels} 
\label{s:reg-01}
Let $m:\R^d\to \R$ be a \emph{regularizer}, i.e.~a bounded function which is $0$ in a neighborhood of $0$ and $1$ in a neighborhood of $\infty$. Define the \emph{regularized kernel} $K_\e$ by $K_\e(x,y)=K(x,y) m((x-y)/\e)$. 
The regularized kernels $K_\e$ are in $L^2\ti{loc}(\mu\times \nu)$ so the regularized integral operators $T_\e$, 
\[
T_\e f(x) :=\int K_\e(x,y) f(y) d\mu(y)
\] 
are well defined for bounded compactly supported $f$. 
In particular, if $m(x) = \1\ci{(1, \infty)}(|x|)$ then we get the classical truncation 
\begin{align}
\label{trunc-01}
T_\e f(x) =\int_{|x-y|>\e} K(x,y) f(y) d\mu(y).
\end{align}

If for $1<p<\infty$ operators $T_\e :L^p(\mu)\to L^p(\nu)$ are uniformly bounded, then by taking w.o.t.~limit point as $\e\to 0^+$ we conclude that $K$ is a kernel of a singular integral operator  (in the sense of the axiomatic approach) with kernel $K$, acting $L^p(\mu)\to L^p(\nu)$. 

It turns out that the converse statement is true, even in a stronger sense, if we  assume that  the measures $\mu$ and $\nu$ do not have common atoms. Namely, the following theorem holds, see \cite[Proposition 2.12]{Regularizations2013}. 


\begin{theo}
\label{t:reg-bd-01}
Let a kernel $K$ be  $L^p(\mu)\to L^p(\nu)$ restrictedly bounded, and assume that $\mu$ and $\nu$ do not have common atoms. Then for any regularizer $m\in C^\infty$ the regularized operators $T_\e$ are uniformly (in $\e$) bounded, $\|T_\e\|\ci{L^p(\mu)\to L^p(\nu)} \le C(m)<\infty$. 
\end{theo}

Moreover, for all ``interesting'' kernels the $L^p(\mu)\to L^p(\nu)$ restricted boundedness implies the uniform boundedness of the classical truncations \eqref{trunc-01}.

Without going into details, we just mention that the ``interesting'' kernels include  kernel $1/(x-y)$, $x,y\in \R$ of the Hilbert transform, the kernel $(x-y)/|x-y|^{\alpha +1}$, $\alpha>0$, $x,y\in\R^d$ of the generalized Riesz transform $R_\alpha$ in $\R^d$,  the kernel  $1/(z-w)$, $z, w\in \C$ of the Cauchy transform, the kernel $1/(z-w)^2$, $z, w\in \C$ of the Beurling--Ahlfors transform and many others. 

Regularizations with smooth functions $m$ seem to be a more logical and convenient choice, than the classical one; for example if one starts with a \cz kernel then after smooth regularizations the resulting kernel will still be a \cz one with uniform estimates of the constants. However, the classical truncations are used most. 

\begin{rem*}
To define truncation of a kernel on the unit circle $\T$ we take the function $m$ on the line, $m\equiv0$ in a neighborhood of $0$ and $m\equiv 1$ in a neighborhood of $\infty$, and define functions $\wt m_\e$ on $\T$ by 
\[
\wt m_\e(e^{it}) = m(t/\e), \qquad -\pi < t \le \pi. 
\]
Then the regularized kernel $K_\e$ is defined as 
\[
K_\e(z, \xi) = K(z,\xi) m_\e(z/\xi), \qquad z, \xi \in \T. 
\]
The regularized kernels on $\T^d$ are defined similarly, and the same results as in $\R^d$ holds in $\T^d$.  
\end{rem*}

\begin{rem}
\label{remark}
For singular integrals  related to complex analysis there is another type of natural regularization. Namely for  the kernel $K(x,y)= 1/(x-y)$ on $\R$ one can consider kernels 
\begin{align}
\label{Cauchy-reg-01}
K_{\pm\e}(x,y) = 1/(x-y\pm i\e).
\end{align}
Similarly, for the kernel 
$
K(z,\xi) = 1/(1-\overline \xi z)
$
on $\T$ define the regularized kernel 
\begin{align}
\label{Cauchy-reg-02}
K_r(z, \xi) =  1/(1-r\overline \xi z), \qquad 0\le r<\infty\quad r\ne 1  . 
\end{align}

For these kernels  Theorem \ref{t:reg-bd-01} holds as well. 
\end{rem}

Now let us discuss the main ideas of the proofs. 

\subsection{First step: Schur multipliers}
The first idea is very simple: we want to multiply a restrictedly bounded kernel by a function $M$ such that the resulting kernel is still restrictedly bounded. 

\begin{defn}
We call a function $M(\fdot, \fdot)$ an $L^p(\mu)\to L^p(\nu)$ Schur multiplier if for any $L^p(\mu)\to L^p(\nu)$ restrictedly bounded kernel $K$ the kernel $KM $ is also $L^p(\mu)\to L^p(\nu)$ restrictedly bounded and 
\[
[KM]\ci{L^p(\mu)\to L^p(\nu)}\ut{r} \le C [K]\ci{L^p(\mu)\to L^p(\nu)}\ut{r}. 
\]
The best constant $C$ in the above inequality is called the Schur norm of $M$. 
\end{defn}

Traditionally, Schur multipliers are defined with respect to the operator norm of the corresponding integral operators, or with respect to the Schatten--von-Neumann norm, but our definition is very close in spirit, so we use the same term. 

Formally, our definition depends on $\mu$, $\nu$ and $p$, but we will construct  ``universal'' multipliers, that work for all $\mu$, $\nu$ and $p$ with the same estimate on the Schur norm. They also are Schur multipliers with respect to the operator norm, as well as with respect to the Schatten--von-Neumann norms.

Thus, in what follows we will omit $L^p(\mu)\to L^p(\nu)$ and simply say Schur multiplier.

\subsubsection{Constructing Schur multipliers via Fourier transform}
We start with an elementary observation: the function $M_a$, $M_a(x, y) :=e^{-i a\cdot x} e^{i a\cdot y} $, $a, x, y\in\R^d$ is a Schur multiplier with the Schur norm $1$ (as a product of two unimodular functions of one variable). 

Averaging in $a$ we get that if $\sigma$ is a complex-valued measure of bounded variation and $m=\widehat \sigma$ is its Fourier transform, 
\[
\widehat \sigma (s) := \int_{\R^d} e^{-is\cdot t} d\sigma (t) 
\]
then the function $M(x,y) = m(x-y)$ is a Schur multiplier with the Schur norm at most $\var \sigma$. 

Note also that for $m_\e(s) = m(s/\e)$  and the measure $\sigma_\e$ defined by $\sigma_\e(E) = \sigma(\e E)$ we have $m_\e= \widehat \sigma_\e$. Since $\var\sigma_\e=\var \sigma$ we get that all the functions $M_\e$
\[
M_\e(x,y) = m_\e(x-y) = m((x-y)/\e)
\]
are Schur multipliers with the Schur norm estimated by $\var\sigma$. 

Since a compactly supported $C^\infty$ function is a Fourier transform of an $L^1$ function (it is a Fourier transform of a Schwartz class function), and $1$ is trivially a Schur multiplier,  we can conclude that functions $M_\e$, $M_\e(x,y) = m((x-y)/\e)$ where $m$ is the $C^\infty$ regularizer defined in Section \ref{s:reg-01}. 

So we see that the regularized kernels $K_\e$ obtained using smooth regularizers $m$ are restrictedly bounded with the uniform (in $\e$) estimate on the restricted norm. 

To get the corresponding result for the torus $\T^d$ we just need to restrict the regularizers $m_\e$ to the cube $(-\pi, \pi]^d$ and then map the cube to the torus via the standard map.

\subsection{Cauchy type regularizations} Let us now discuss the Cauchy type regularizations \eqref{Cauchy-reg-01}, \eqref{Cauchy-reg-02}. For $\rho(x) =\1\ci{[0,\infty)} e^{-x}$ define 
\[
m(s)= 1-\widehat \rho(s) = \frac{s}{s-i}. 
\]
Then $m_\e(s) = m(s/e) = {s}/({s-i\e})$, and the functions $M_\e(x,y) = m_\e(x-y)$ are Schur multipliers with Schur norm at most $2$. Computing the regularized kernel we get 
\[
K_\e(x,y) = \frac{1}{x-y} \frac{x-y}{x-y-i\e} = \frac{1}{x+i\e -y}, 
\]
so the kernels $K_{+\e}$ from \eqref{Cauchy-reg-01} are uniformly restrictedly bounded. 

Repeating the same reasoning with $\rho(x) =\1\ci{(-\infty, 0]} e^{x}$ we get the conclusion for $K_{-\e}$. 

For the kernel \eqref{Cauchy-reg-02} on $\T$ we use the Fourier transform on $\Z$. Namely, it is easy to show that if $a\in\ell^1(\Z)$ and $m(z) := \sum_{k\in \Z} a_k z^k$, $z\in\T$,   then the function $M$
\[
M(z,\xi) = m(z/\xi) \qquad z, \xi \in \T
\]
is a Schur multiplier with Schur bound at most $\|a\|\ci{\ell^1}$. 

Then for $0\le r <1$ multiplying $K(z,\xi) = 1/(1-\overline \xi z)$ by
\[
m(z/\xi)= 1 + \sum_{n=1}^\infty (r^n- r^{n-1}) (\overline \xi z)^n = \frac{1-\overline\xi z}{1-r\overline \xi z}
\]
we at most double the restricted norm (because $1+\sum_{n=1}^\infty |r^n-r^{n-1}| = 1+r\le 2$). 
So, for the kernel
\[
K(z, \xi) \cdot \frac{1-\overline\xi z}{1-r\overline \xi z} = \frac{1}{1-r\overline \xi z} =K_r(z,\xi), \qquad r<1
\]
we get for $r<1$
\begin{align}
\label{Cuachy-restr-est-02}
[K_r]\ci{L^p(\mu)\to L^p(\nu)} \le 2 [K]\ci{L^p(\mu)\to L^p(\nu)}. 
\end{align}

For $r>1$ we can write
\begin{align*}
m(z/\xi)=\frac{1-\overline\xi z}{1-r\overline \xi z} = 1-\sum_{n=1}^\infty(r^{-n} - r^{-(n+1)} ) (\overline \xi z)^{-n}.
\end{align*}
Noticing that $1+\sum_{n=1}^\infty |r^{-n} - r^{-(n+1)}| = 1 + r^{-1}\le 2$ we see that in the case $r>1$ 
\eqref{Cuachy-restr-est-02} holds as well.

\subsection{Final step: boundedness of the regularized operators}

\begin{theo}
\label{t:restr-bd1}
Let $\mu$ and $\nu$ be Radon measures in $\R^d$ 
without common atoms. Assume that a kernel $K\in L^2\ti{loc} (\mu\times \nu)$ is $L^p(\mu)\to L^p(\nu)$ restrictedly bounded, with the restricted norm $C$. Then the integral operator with $T$ kernel $K$ is a bounded operator $L^p(\mu)\to L^p(\nu)$ with the norm at most $2C$. 

\end{theo}

Restricting the kernels to compact subsets exhausting $\R^d\times\R^d$ one can easily reduce the proof to the case $K\in L^2 (\mu\times \nu)$ (globally, not locally). Then the idea of the proof is very simple. Taking bounded compactly supported functions $f$ and $g$ we can write
\[
\La T f, g\Ra_\nu = \int K(x,y) f(y) g(x) d\mu(y) d\nu(x) .
\]

The main idea of the proof is to construct bounded functions $f_n$, $g_n$ with separated compact supports such that $f_n \rightharpoonup \frac12 f$ weakly in $L^2(\mu)$, $g_n \rightharpoonup \frac12 g$ weakly in $L^2(\nu)$ and such that 
\begin{align}
\label{lim-norm}
\limsup_{n\to\infty} \| f_n\|\ci{L^p(\mu)} \le 2^{-1/p} \|f\|\ci{L^p(\mu)}, \qquad 
\limsup_{n\to\infty} \| g_n\|\ci{L^{p'}(\mu)} \le 2^{-1/p'} \|g\|\ci{L^{p'}(\nu)}\,.
\end{align}
Since the operator $T$ is Hilbert--Schmidt, and so compact (as an operator $L^2(\mu)\to L^2(\nu)$) the weak convergence implies that 
\[
\La T f_n , g_n\Ra_\nu \to \frac14 \La Tf,g\Ra_\nu. 
\] 
Therefore, using \eqref{lim-norm} we get 
\[
|\La Tf,g\Ra | \le \limsup_{n\to\infty} 4|\La Tf_n,g_n\Ra| \le 2 C\|f\|\ci{L^p(\mu)} \|g\|\ci{L^{p'}(\nu)}.
\]

The man idea of the construction of the functions$f_n$ and $g_n$ is quite simple, at least for the absolutely continuous piece: we define $f_n:= \1\ci{E_n}$, $g_n:=\1\ci{F_n}$ where $E_n$ and $F_n$ are separated  ``mesh like'' subsets, that are well mixed, meaning that that for all dyadic cubes $Q$ of size at least $2^{-n}$ the Lebesgue measure of the sets $Q\cap E_n$ and $Q\cap F_n$ is almost half (with relative error of say $2^{-n}$) of the measure of $Q$. Construction of such sets in for the Lebesgue measure is rather trivial and can be left as an exercise for the reader. 

For the measures $\mu$ and $\nu$ without atoms the construction is almost the same, only the ``well mixed'' property is with respect to the measure $\sigma =\mu+\nu$, meaning that for any dyadic cube $Q$ of size at least $2^{-n}$ the measures $\sigma(Q\cap E_n)$, $\sigma(Q\cap F_n)$ are almost half of $\sigma(Q)$ with relative error $2^{-n}$. It might not be immediately obvious how to construct such sets $E_n$, $F_n$, but the construction is relatively simple and straightforward, see 
 \cite{Regularizations2013} for details. 
 
 The construction in the general case is just a bit more complicated. Namely, we first construct the sets $E_n$ and $F_n$ with respect to the continuous parts $\mu\ti c$, $\nu\ti c$ of the measures (making sure that the sets do not contain any atoms). Then we define $f_n$ and $g_n$ by adding to $f\1\ci{E_n}$ and $g\1\ci{F_n}$ the functions
 \[
 \frac12 \sum_{k=1}^n f(a_k) \delta_{a_k}, \qquad 
 \frac12 \sum_{k=1}^n g(b_k) \delta_{b_k}
 \]	
respectively, where $a_k$, $b_k$ are atoms of $\mu$ and $\nu$ respectively. To make sure that the functions $f_n$ and $g_n$ have separated supports, we then just need to ``shrink'' the  sets $E_n$ $F_n$ by removing small discs around atoms. Again, the reader is referred to \cite{Regularizations2013} for the details. 

This idea of using ``well mixed'' set was exploited in \cite{NikTr2002} in the case of Lebesgue measure. It was later used in \cite{LiawTreil2009}, where some of the result in this section were proved under the assumption that the singular parts of $\mu$ and $\nu$ are mutually singular. 

The results in full generality were proved in \cite{Regularizations2013}, the reader should look there for full details.

\section{Clark theory for rank one perturbations of unitary operators}\label{s-UNITARY}

\subsection{Plan of the game}
As we discussed above in Section \ref{s:UnitRk1}, rank one unitary perturbations of a unitary operator $U$ are parametrized by the formula \eqref{U_gamma-SA}. If in \eqref{U_gamma-SA} we take $|\alpha|<1$ (instead of $|\alpha|=1$) the resulting operator $U_\alpha$ will be not a unitary, but only a contractive ($\|U_\alpha\|\le1$) operator. 

If, as in Section \ref{s:UnitRk1} we assume by ignoring the trivial part that $b$ is $*$-cyclic vector for $U$, then for $|\gamma|<1$ the operator $U_\gamma = U+ (\gamma-1)bb_1^*$, $b_1=U^*b$ is a \emph{completely non-unitary} (c.n.u.) contraction. The term completely non-unitary means that there is no reducing (i.e.~invariant for $U_\gamma$ and $U^*_\gamma$) subspace on which $U_\gamma$ acts unitarily. 

A completely non-unitary contraction $T$ is up to unitary equivalence determined by its so-called \emph{characteristic function} $\theta=\theta\ci T$, see the definition below. Namely, $T$ is unitarily equivalent to its \emph{model} $\cM=\cM_\theta$, where $\cM_\theta$ is a \emph{compression}  of the multiplication operator $M_z$, 
\begin{align*}
\cM_\theta  = P_{\theta} M_z \bigm|_{\textstyle \cK_\theta};
\end{align*}
here $\cK_\theta $ is a subspace of a generally vector-valued, and possibly weighted $L^2$ space on the unit circle, $P_\theta = P_{\cK_\theta}$ is the orthogonal projection onto $\cK_\theta$, and $M_z$ is the multiplication by the independent variable $z$, $M_z f(z) = z f(z)$, $z\in \T$. 

So, we have two unitarily equivalent representations of the operator $U_\gamma$, $|\gamma|<1$: the representation 
\[
U_\gamma = M_\xi + (\gamma-1)bb_1^*, \qquad b=\1, \quad b_1=M_\xi^* \1
\]
in the spectral representation of $U$ in $L^2(\mu)$, where $\mu$ is the spectral measure of $U$ corresponding to the vector $b$, and the representation as the model operator $\cM_{\theta_\gamma}$ in the model subspace $\cK_{\theta_\gamma}$. 

The Clark theory describes the unitary operator providing this unitary equivalence, i.e.~a unitary operator $\Phi_\gamma: \cK_{\theta_\gamma} \to L^2(\mu)$ such that 
\[
\Phi_\gamma \cM_{\theta_\gamma} = U_\gamma \Phi_\gamma. 
\]
D.~Clark in his original paper \cite{Clark} described such operators for the particular case when $\theta_\gamma$ is an inner function. He started with   the model operator (unitarily equivalent to $U_\gamma$, $|\gamma|<1$ in our notation) in a particular case of inner characteristic function, described all its unitary rank one perturbations ($U_\alpha$, $|\alpha|=1$ in our notation) and described the unitary operator between the model operator $\cM_\theta$ and the spectral representation  of $U_\alpha$, $|\alpha|=1$. 

Translated to our language the fact that the characteristic function $\theta$ is inner means  that  the operator $U$ (and so all $U_\alpha$, $|\alpha|=1$) have purely singular spectrum.

\subsection{A functional model for a c.n.u.~contraction}
Let us recall the definition related to the functional model. For an operator $T$ acting in a separable Hilbert space we define the defect operators 
\[
D\ci T:= (\bI -T^*T)^{1/2}, \qquad D\ci{T^*}:=(\bI-TT^*)^{1/2},   
\]
and the defect subspaces
\begin{align*}
\fD=\fD_{T} := \clos \Ran D_T, \qquad \fD_*=\fD_{T^*} := \clos\Ran D_{T^*}.
\end{align*}
The characteristic function $\theta=\theta_T$ of the operator $T$ is an analytic function $\theta =\theta_T \in H^\infty_{\fD\shto \fD_*}$ whose values are bounded operators (in fact, contractions) acting from $\fD$ to $\fD^*$ defined by the equation 
\begin{align}
\label{CharFunction}
\theta_T(z) = \left( -T + z D_{T^*}(\bI - z T^*)^{-1} D_T \right) \Bigm|_{\textstyle\fD}, \qquad z\in \D. 
\end{align}
Note that $T\fD \subset \fD_*$, so for $z\in D$ the above expression indeed can be interpreted as an operator from $\fD$ to $\fD_*$. 

It is customary to assume that the characteristic function is defined up to constant unitary factors on the right and on the left, i.e.~one considers the whole equivalence class consisting of functions $U\theta V$, 
where $U:\fD_*\to E_*$ and $V: E\to \fD$ are unitary operators and $E_*$, $E$ are Hilbert spaces of appropriate dimensions. 
The advantage of this point of view is that we are not restricted to using the defect spaces of $T$, but can work with arbitrary Hilbert spaces of appropriate dimensions. 

Note, that the characteristic function (defined up to constant unitary factors) is a unitary invariant of a completely 
non-unitary contraction: any two such contractions with the same characteristic function are unitarily equivalent.

Note also, that given a characteristic function, any representative  gives us a model, and there is a standard unitary equivalence between the model for different representatives.

\begin{rem*}
Another way to look at a choice of a representative of a characteristic function is to pick orthonormal bases in the defect spaces and treat the characteristic function as a matrix-valued function (possibly of infinite size). The choice of the orthonormal bases is equivalent to the choice of the constant unitary factors. 
\end{rem*}

In this paper by a \emph{functional model} associated to an operator-valued function $\theta\in H^\infty_{E\shto E_*}$ we understand the following: a model space $\cK_\theta$ is an appropriately constructed subspace of a (possibly) weighted space $L^2(E_*\oplus E,W)$ on the unit circle $\T$ with the operator-valued weight  $W$.  The model operator $\cM_\theta$ is a compression of the multiplication operator $M_z$ onto $\cK_\theta$, 
\begin{align}
\label{model-02}
\cM_\theta  = P_{\theta} M_z \bigm|_{\textstyle \cK_\theta};
\end{align}
where $P_\theta = P_{\cK_\theta}$ is the orthogonal projection onto $\cK_\theta$. 

All the functional models for the same $\theta$ are unitarily equivalent, so sometimes people interpret them as different \emph{transcriptions} of one object. 

As we already mentioned above, a completely non-unitary contraction with characteristic function $\theta$ is unitarily equivalent to its model $\cM_\theta$. 

On the other hand, for any purely contractive $\theta\in H^\infty\ci{E\shto E_*}$, 
$\|\theta\|_\infty\le 1$ the model operator $\cM_\theta$ is a completely non-unitary contraction, with $\theta $ being its characteristic function. Thus, any such $\theta$ is a characteristic function of a completely non-unitary contraction.


\subsubsection{Sz.-Nagy--
\texorpdfstring{Foia\c{s}}{Foias} 
transcription}

The Sz.-Nagy--Foia\c{s} model (transcription) is \linebreak 
probably the most used.

The model  space $\K_\theta$ is defined as a subspace of $L^2(E_*\oplus E)$ (non-weighted, $W(z)\equiv \I$), 
\begin{align}\label{K_theta}
 \K_\te=
 \left( \begin{array}{c} H^2_{E_*} \\ \clos\Delta L^2_{E} \end{array}\right) 
 \ominus 
  \left( \begin{array}{c} \theta \\ \Delta \end{array}\right) H^2_{E}
\intertext{where the defect $\Delta$ is given by}
\label{d-Delta}
 \Delta(z) := (1-\theta(z)^*\theta(z))^{1/2} , \qquad z\in \T
 .
\end{align}
If the characteristic function $\theta$ is \emph{inner}, meaning that its boundary values are isometries a.e.~on $\T$, then $\Delta \equiv 0$, so the lower ``floor'' of $\cK_\theta$ collapses and we get a simpler, ``one-story'' model subspace, 
\begin{align*}
\cK_\theta = H^2(E_*)\ominus \theta H^2(E). 
\end{align*}
This subspace is probably much more familiar to analysts, especially when $\theta$ is a scalar-valued function. 

The model operator $\cM$ is defined by \eqref{model-02} as the compression of the multiplication operator $M_z$ (also known as forward shift operator) onto $\cK_\theta$, 
and the multiplication operator $M_z$ is understood as the entry-wise multiplication by the independent variable $z$, 
\[
M_z \left(\begin{array}{c} g \\ h \end{array} \right) = \left(\begin{array}{c} zg \\ zh \end{array} \right). 
\]

As we discussed above, 
the characteristic function $\theta$ is defined up to constant unitary factors on the right and on the left. 
But one has to be a bit careful here, because if $\wt \theta(z) = U \theta (z) V$, where $U$ and $V$ are constant unitary operators, then the spaces $\cK_\theta$ and $\cK_{\wt\theta}$ are different. 

However, the map $\cU$
\[
\cU \left(\begin{array}{c} g \\ h \end{array} \right) = \left(\begin{array}{c} U g \\ V^*  h \end{array} \right) 
\]
is the canonical unitary map transferring the model from one space to the other. 

Namely, it is easy to see that $\cU$ is a unitary map from $H^2(E_*) \oplus \clos\Delta L^2 (E)$ onto $H^2(UE_*) \oplus \clos \wt\Delta L^2 (V^* E)$, where $\wt \Delta = \Delta_{\wt \theta} = V^*\Delta V$. Moreover, it is not difficult to see that $\cU \cK_\theta = \cK_{\wt\theta}$ and that $\cU$ commutes with the multiplication by $z$, so $\cU_\theta := \cU\bigm|_{\textstyle \cK_\theta}$ intertwines the model operators, 
\[
\cU_\theta \cM_\theta = \cM_{\wt\theta} \cU_\theta  .
\]

\subsubsection{de Branges--Rovnyak transcription}
\label{s:deBrangesRepr}
Let us present this transcription as it is described in \cite{Nik-Vas_model_MSRI_1998}. Since the ambient space in this transcription is a weighted $L^2$ space with an operator-valued weight, let us recall that if $W$ is an operator-valued weight on the circle, i.e.~a function whose values are self-adjoint non-negative operators in a Hilbert space $E$, then the norm in the space $L^2(W) $ is defined as 
\[
\|f\|_{L^2(W)} = \int_\T \left( W(z) f(z), f(z)\right)\ci{E} \frac{|dz|}{2\pi} \,.
\]
There are some delicate details here in defining the above integral if we allow the values $W(z)$ to be unbounded operators, but we will not discuss it here. In our case when the characteristic function is scalar-valued the values $W(z)$ are bounded self-adjoint operators on $\C^2$, and the definition of the integral is straightforward. 

Let 
\[
W_\theta(z) = \left(\begin{array}{cc} \I & \theta(z) \\ \theta(z)^* & \I \end{array}\right)\, .
\]
The weight in the ambient space will be given by $W=W_\theta^{[-1]}$, $W_\theta^{[-1]}(z) = (W_\theta(z))^{[-1]}$ where 
$A^{[-1]}$ stands for the Moore--Penrose inverse of the operator $A$. If $A=A^*$ then $A^{[-1]}$ is $\OZ$ on $\Ker A$ and is 
equal to the left inverse of $A$ on $\Ran A$. The model space
$\cK_\theta$ is defined as
\begin{align}
\label{deBrangesRepr}
\cK_\theta = \left\{
\left(\begin{array}{c} g_+ \\ g_-\end{array} \right) \,:\ g_+\in H^2(E_*),\  g_- \in H^2_-(E),\  g_- - \theta^* g_+ \in \Delta L^2(E)
\right\}   .
\end{align}

\begin{rem}
\label{r:dBr-R_orig}
The original de Branges--Rovnyak model was initially described in \cite{deBr-Rovn_CanonModels_1966} using  completely different terms. To give the definition from \cite{deBr-Rovn_CanonModels_1966} we need to recall the notion of a Toeplitz operator. For $\f\in L^\infty_{E\shto E_*}$ the Toeplitz operator $T_\f : H^2(E)\to H^2(E_*)$ with symbol $\f$ is defined by 
\[
T_\f f := P_+ (\f f), \qquad f\in H^2(E). 
\]

The (preliminary) space $\cH(\theta) \subset H^2(E_*)$ is defined as a range $(\I - T_{\theta}T_{\theta^*})^{1/2} H^2 (E)$ endowed with the \emph{range norm} (the minimal norm of the preimage). 

Let the involution operator $J$ on $L^2(\T)$ be defined as 
\[
Jf(z) = \overline z f(\overline z). 
\]
Following de Branges--Rovnyak \cite{deBr-Rovn_CanonModels_1966} define  the \emph{model space} $\cD(\theta)$  as the set of vectors 
\[
\left( \begin{array}{c}
g_1 \\ 
g_2 \\ 
\end{array} \right) \ : g_1\in \cH(\theta), \ g_2\in H^2(E),  \text{ such that } z^n g_1 -\theta P_+(z^n Jg_2) \in \cH(\theta) \ \forall n\ge 0, 
\] 
and such that 
\[
\left\| \left( \begin{array}{c}
g_1 \\ 
g_2 \\ 
\end{array} \right) \right\|_{\cD(\theta)}^2 := \lim_{n\to\infty}
\left( \|z^n g_1 -\theta P_+(z^n Jg_2)\|\ci{\cH(\theta)}^2 + \|P_+(z^n Jg_2)\|_2^2 \right) <\infty. 
\]
It might look surprising, but it was proved in \cite{Nik-Vas_FunctModels_1989} that the operator 
$\left( \begin{array}{c}
g_+ \\ 
g_- \\ 
\end{array} \right)
\mapsto 
\left( \begin{array}{c}
g_+ \\ 
Jg_- \\ 
\end{array} \right)
$
is a unitary operator between the described above model space $\cK_\theta$ in the de Branges--Rovnyak transcription  and the model space $\cD(\theta)$. 
\end{rem}

\subsection{Model for the operator \texorpdfstring{$U_\gamma$}{U<sub>gamma}.} For the perturbations $U_\gamma$, $|\gamma|<1$ the functional model can be computed explicitly.

The defect operators are computed to be
\begin{align*}
D_{U_\gamma} &= \left(  \mathbf{I} - {U_\gamma}^*{U_\gamma} \right)^{1/2} = \left( 1 -|\gamma|^2\right)^{1/2} b_1b_1^*,\\
 D_{{U_\gamma}^*} &= \left( \mathbf{I}- {U_\gamma}{U_\gamma}^*\right)^{1/2} = \left( 1 -|\gamma|^2\right)^{1/2} b b^*
\end{align*}
and the 
defect spaces are
$$
\mathfrak{D}=\mathfrak{D}\ci{U_\gamma} = \text{span}\{b_1\}\qquad\text{and}\qquad
\mathfrak{D}_\ast =\mathfrak{D}\ci{{U_\gamma}^*}= \text{span}\{b\}.
$$
Note that the defect spaces are one-dimensional, so the characteristic function $\theta=\theta_\gamma$  is a scalar-valued function. We already mentioned above that $\theta \in H^\infty$, $\|\theta\|_\infty<1$. 
Note also that the defect spaces do not depend on $\gamma$. 

%
%

The characteristic function $\theta_\gamma$ of $U_\gamma$ can be computed in terms of Cauchy type transforms. 
For a (possibly complex-valued) measure $\tau$ on $\T$ and $\la\notin\T$ define the Cauchy type transforms $R$, $R_1$ and $R_2$ by
\begin{align}
\label{CauchyTrans}
R \tau (\la) := \int_\T \frac{d\tau(\xi)}{1-\overline\xi\la}, \quad R_1 \tau (\la) := \int_\T \frac{\overline\xi \la d\tau(\xi)}{1-\overline\xi\la}, \quad R_2\tau(\la):= \int_\T \frac{1+ \overline\xi \la }{1-\overline\xi\la} d\tau(\xi).
\end{align}
If we pick $b_1$ and $b$ to be the basis vectors in the corresponding defect spaces, then the characteristic function $\theta_\gamma$ of the operator $U_\gamma$, $|\gamma|<1$ is given by 
\begin{align}\label{e-CHARgamma}
\theta_\gamma(\la) =  -\gamma + \frac{(1-|\gamma|^2) R_1 \mu(\la)}{1 + (1-\overline\gamma) R_1 \mu(\la)} = 
\frac{(1-\gamma) R_2\mu(\la) -(1+\gamma)}{(1-\overline\gamma )R_2\mu(\la)+(1+\overline\gamma)}\,, \qquad \la\in \D.
\end{align}
Note that the formulas for $\theta_0$ ($\gamma=0$) are especially simple. 
And $\te_0$ is related to $\te_\gamma$ by a fractional transformation:
\begin{align}
\label{theta_0-theta_gamma}
\theta_\gamma = \frac{\theta_0-\gamma}{1-\overline\gamma\theta_0} \qquad
\textup{or equivalently } \qquad \theta_0 = \frac{\theta_\gamma + \gamma}{1 + \overline\gamma\theta_\gamma} \,.
\end{align}

To compute the characteristic function one can use the definition \eqref{CharFunction} of the characteristic function with $U_\gamma$ instead of $T$ and the inversion formula \eqref{pert_I}. Namely, writing 
\begin{align*}
\I - z U_\gamma^* = (\I - z U^*)\left(I- z(\overline \gamma -1) (\I-zU^*)^{-1} b_1b^*\right)
\end{align*}
and applying the inversion formula \eqref{pert_I} we get denoting $\beta=\gamma-1$
\begin{align*}
(\I - z U_\gamma^*)^{-1} = \left( \I +  \frac{1}{\left(z\overline \beta (\I-zU^*)^{-1} b_1, b\right)\ci{\cH}} z\overline \beta (\I-zU^*)^{-1}b_1b^* \right) (\I -zU^*)^{-1}.
\end{align*}
In the spectral representation of $U$ in $L^2(\mu)$ the operator $(I-zU^*)^{-1}$ is the multiplication by the function $1/(1-\overline \xi z)$, $b\equiv\1$, $b_1(\xi)\equiv \xi$, so the above inverse can be explicitly computed. Then standard algebraic manipulations lead to the formulas  \eqref{e-CHARgamma} for the resolvent. 

A different way of computing the characteristic function for finite rank perturbations can be found in \cite{RonRkN}. 

We point out that if the measure $\mu$ is purely singular (with respect to the Lebesgue measure), then the functions $\theta_\gamma$ are inner ($|\theta_\gamma|=1$ a.e.~on $\T$). In this case the model is especially simple, the model space consists of scalar functions,  and that is the case treated by the original Clark theory. 

However, in our case, $\mu$ is an arbitrary probability measure, so the characteristic functions can be non inner, and the model is more complicated: the model space consists of vector-valued functions (with values in $\C^2$).

\subsection{Preliminaries about Clark operator}
Recall that our goal is to describe a \emph{Clark operator}, i.e.~a unitary operator (non-uniqueness is discussed in the next paragraph) that realizes unitary equivalence between $U_\gamma$ and $\cM_{\te_\gamma}$. Namely, we want to find a unitary operator $\Phi_\gamma:\cK_{\theta_\gamma}\to L^2(\mu)$
such that 
  \begin{align}\label{Phi*-intertwine}
  \Phi_\gamma \cM_{\te_\gamma}  = U_\gamma \Phi_\gamma,
  \end{align}

Let us discuss what freedom do we have in choosing such an operator. Clearly, $\Phi_\gamma$ maps defect spaces of $\cM_\gamma$ to the corresponding defect spaces of $U_\gamma$. Therefore, $\Phi^*_\gamma b$ and $\Phi^*_\gamma b_1$ must be unit vectors in $\mathfrak{D}_{\cM_{\te_\gamma}^\ast}$ and $\mathfrak{D}_{\cM_{\te_\gamma}}$ respectively. 

We say that the unit vectors $c\in \mathfrak D_{\cM_{\theta_\gamma}^*}$ and $c_1\in \mathfrak D_{\cM_{\theta_\gamma}}$ \emph{agree} if there exists a unitary map $\Phi_\gamma:\cK_{\theta_\gamma}\to L^2(\mu)$ satisfying \eqref{Phi*-intertwine} such that 
\[
\Phi_\gamma^* b = c, \qquad \Phi_\gamma^* b_1 =c_1 .
\]
If $\gamma = 0$ and $\mu$ is the Lebesgue measure, then it is not hard to see that $\te_\gamma\equiv 0$. 
It is also easy to see that in this case, any two unit vectors $c\in \mathfrak D_{\cM_{\theta_\gamma}^*}$ and $c_1\in \mathfrak D_{\cM_{\theta_\gamma}}$ agree.

Otherwise, if either $\gamma\neq 0$ or $\mu$ differs from the Lebesgue measure, then for any unit vector $c\in \mathfrak D_{\cM_{\theta_\gamma}^*}$ there exist a unique vector $c_1\in \mathfrak D_{\cM_{\theta_\gamma}}$ which agrees with $c$; for details see Proposition 2.9 of \cite{LT}.  That means the operator $\Phi_\gamma$ is unique up to a multiplicative unimodular constant $\alpha\in\T$; in particular, if we fix a unit vector $c \in \mathfrak D_{\cM_{\theta_\gamma}^*}$ then the condition $\Phi_\gamma c =b$ uniquely determines the Clark operator $\Phi_\gamma$. 

In the trivial case when $\mu$ is the normalized Lebesgue measure and $\gamma=0$ the Clark operator $\Phi_\gamma$ can be easily constructed via elementary means, so in what follows we will ignore this case.

%

\subsection{A \texorpdfstring{``universal"}{"universal"} representation formula for the adjoint of the Clark operator}\label{ss-Phi}


An explicit computation of the defect spaces of the compressed shift operator $\cM_\te$ yields that in the
Sz.-Nagy--Foia\c s transcription
\begin{align*}
\mathfrak D_{\cM^*_{\te}} = \spa\{c\}, \qquad \mathfrak D_{\cM_{\te}} = \spa\{ c_1\}, 
\end{align*}
where 
\begin{align}
\label{c}
c(z) &:= \left( 1- |\theta(0)|^2 \right)^{-1/2} 
\left( \begin{array}{c} 1-\overline{\theta(0)}\theta (z)\\ -\overline{\theta(0)}\Delta(z) \end{array} \right), 
\\
\label{c_1}
c_1(z) & := \left( 1- |\theta(0)|^2 \right)^{-1/2} 
\left( \begin{array}{c} z^{-1} \left(\theta(z)- \theta(0)\right) \\ z^{-1} \Delta (z)\end{array} \right),
\end{align}
where $\Delta:= (1-|\theta|^2)^{1/2}$. 

Moreover, the vectors $c$ and $c_1$ are of unit length and agree. Note also that it follows from \eqref{e-CHARgamma} that $\theta_\gamma(0)=-\gamma$, so the above formulas can be further simplified.


The following theorem describes the adjoint $\Phi_\gamma^*$ of the Clark operator. Note that the intertwining relation \eqref{Phi*-intertwine} can be rewritten as 
\[
\Phi_\gamma^* U_\gamma = \cM_{\theta_\gamma} \Phi_\gamma^*.
\] 

\begin{theo}[A ``universal" representation formula; Theorem 3.1 of \cite{LT}]
\label{t-reprB}
Let $\theta_\gamma$  be a characteristic function (one representative) of $U_\gamma$,  $|\gamma|<1$, and let $\cK_{\theta_\gamma}$ and $\cM_\gamma =\cM_{\theta_\gamma}$ be the model subspace and the model operator respectively. 
Assume that the unit vectors $c=c^\gamma\in\mathfrak D_{\cM_{\theta_\gamma}^*}$, $c_1=c_1^\gamma \in \mathfrak D_{\cM_{\theta_\gamma}}$ agree. 
Let $\Phi^*_\gamma: L^2(\mu) \to \cK_{\theta_\gamma}$ be the unitary operator satisfying 
\[
\Phi_\gamma^* U_\gamma = \cM_{\theta_\gamma} \Phi_\gamma^* , 
\]
and such that $\Phi_\gamma^* b = c^\gamma$, $\Phi_\gamma^* b_1 = c_1^\gamma$. 

Then for all $f\in C^1(\T)$
\begin{align}
\label{f-reprB}
\Phi_\gamma^* f(z)&= A_\gamma(z) f(z) +
B_\gamma(z)\int\frac{f(\xi)-f(z)}{1- \overline\xi z}\, d\mu(\xi)
\end{align}
where $A_\gamma(z) = c^\gamma(z)$, $B_\gamma(z) = c^\gamma(z) - z c_1^\gamma(z)$.
\end{theo}

\begin{proof}[Idea of the proof]
To some extent we mirror the proof of Theorem \ref{t-repr-V-unitary}. However, several miracles occur (beyond the fact that we are now dealing with the vector-valued setting of the model space makes the computations are more cumbersome):

Again, we begin with the intertwining relation $\Phi_\gamma^* U_\gamma = \cM_{\theta_\gamma} \Phi^*_\gamma$ and evaluate the projection of the model operator
\begin{align}
\label{MOD}
\cM_{\theta_\gamma}  & = M_z -zc_1^\gamma (c_1^\gamma)^* - \theta_\gamma(0) c^\gamma (c_1^\gamma)^*  
 = M_z + (\gamma c^\gamma - zc_1^\gamma )(c_1^\gamma)^*  .
\end{align}
We notice that the model operator $ \cM_{\theta_\gamma}$ on $\cK_\te$ is a rank one perturbation of the unitary $M_z$, and the operator $U_\gamma$ on $L^2(\mu)$ is a rank one perturbation of the unitary $U_1$ (multiplication by the independent variable). So we expect that the commutator $\Phi^*_\gamma U_1 -  M_z \Phi_\gamma^*$ is at most of rank 2. But in fact, it turns out to be of rank one!

Indeed, the intertwining relation $\Phi_\gamma^* U_\gamma = \cM_{\theta_\gamma} \Phi^*_\gamma$ can be rewritten as
\begin{align*}
\Phi^*_\gamma U_1 + (\gamma-1) c^\gamma b_1^* = M_z \Phi_\gamma^* + (\gamma c^\gamma -c_2^\gamma) b_1^*
\end{align*}
(here we used that $\Phi^*_\gamma b =c^\gamma$ and $(c_1^\gamma)^* \Phi_\gamma^* = (\Phi_\gamma c_1^\gamma)^* = b_1^*$), and therefore
\begin{align}
\label{COMM}
\Phi_\gamma^* U_1 = M_z \Phi_\gamma^*  + (c^\gamma-zc_1^\gamma) b_1^*. 
\end{align}

From here, we proceed in analogy to the proof of \ref{t-repr-V-unitary} to obtain a formula for $\Phi_\gamma^* \xi^n$.

The formula for  $\Phi_\gamma^* \bar\xi^n$ cannot be computed by simply taking the formal adjoint of the commutation relation \eqref{COMM}. This is due to the fact that in general $zc_1^\gamma\notin\cK_\te$. Instead we compute the adjoint of the model operator in analogy to \eqref{MOD}
\begin{align*}
\cM^*_{\theta_\gamma}  & = M_{\overline z} - M_{\overline z} c^\gamma (c^\gamma)^* - \overline{\theta(0)} c_1^\gamma (c^\gamma)^* 
        = M_{\overline z} + (\overline \gamma c_1^\gamma - M_{\overline z} c^\gamma ) (c^\gamma)^* .
\end{align*}

We find ourselves in the lucky situation that the formulas for $\Phi_\gamma^* \xi^n$ and $\Phi_\gamma^* \bar\xi^n$ turn out to be the same.
\end{proof}

In the Sz.-Nagy--Foia\c s transcription we derive concrete formulas.
For $\gamma=0$ we have $\te_0(0) = 0$ and by \eqref{theta_0-theta_gamma} we obtain $\te_\gamma(0) = -\gamma$. With this, the vector-valued functions $A_\gamma(z)$ and $B_\gamma(z)$ in the universal representation formula \eqref{f-reprB} evaluate to
\begin{align}\label{e-AB}
A_\gamma(z) 
& = c^\gamma(z) = (1-|\gamma|^2)^{-1/2} 
\left(\begin{array}{c} 1 + \overline\gamma  \theta_\gamma(z) \\ \overline\gamma \Delta_\gamma(z) \end{array}\right) 
= \left(\begin{array}{c}  \frac{(1-|\gamma|^2)^{1/2}}{1 - \overline\gamma  \theta_0(z)} \\ \frac{\overline\gamma\Delta_0(z)}{| 1- \overline\gamma \theta_0(z)|} \end{array}\right)
,
\\
B_\gamma(z)
 & = c^\gamma(z) - z c_1^\gamma(z)  = (1-|\gamma|^2)^{-1/2} 
\left(\begin{array}{c} 1 + (\overline\gamma  -1)\theta_\gamma(z) -\gamma\\ (\overline\gamma -1)\Delta_\gamma(z) \end{array}\right) \\
\notag
&  \qquad\qquad =
\left(\begin{array}{c}  (1-|\gamma|^2)^{1/2}
{(1-\theta_0(z))}/{(1 - \overline\gamma  \theta_0(z))} \\ (\overline\gamma -1) 
{\Delta_0(z)}/{| 1- \overline\gamma \theta_0(z)|} \end{array}\right) , 
\end{align}
where $\Delta_\gamma=(1-|\theta_\gamma|^2)^{1/2}$.

\subsection{Singular integral operators and a representation for \texorpdfstring{$\Phi_\gamma^*$}{Phi<sup>*} in the Sz.-Nagy--\texorpdfstring{Foia\c s}{Foias} transcription 
}

In this section we get a representation of $\Phi_\gamma^*$ adapted to the  Sz.-Nagy--Foia\c s transcription, similar to the representations given in Theorem \ref{reg-2}.

We  first note that for $v(\xi)=|B_\gamma(\xi)|^2$ the kernel $K(z,\xi) =1/(1-\overline \xi z)$ is an $L^2(\mu)\to L^2(v)$ restrictedly bounded kernel, see Definition \ref{d:RestrBd}. Indeed, taking $C^1$ functions $f$ and $g$ with separated compact supports we get that 
\[
(\cV_\gamma f, g) = \int_\T \frac{\bigl(B(z)f(\xi), g(z)\bigr)\ci\cH}{1-\overline \xi z} d\mu(\xi) \frac{|dz|}{2\pi}, 
\]
and standard approximation reasoning extend this formula to all bounded functions with separated supports. 
But that means that the vector-valued kernel%
\footnote{We did not discuss singular integral operators with vector-valued kernels, but the extension of the theory presented in Section \ref{s-SIO} to the case of kernels with values in $\R^d$ or $\C^d$ is trivial and we omit it.}
$B_\gamma(z)/(1-\overline\xi z)$ is a kernel of a singular integral operator $L^2(\mu)\to L^2$ with norm $1$, and so it is  $L^2(\mu)\to L^2$ restrictedly bounded (with restricted norm at most $1$). 

A standard renormalization argument then implies the $L^2(\mu)\to L^2(v)$ restricted \linebreak
boundedness of the scalar kernel $1/ (1-\overline \xi z) $. 

Therefore, as we discussed in Section \ref{s-SIO}, see Theorem \ref{t:reg-bd-01} and Remark \ref{remark}, the regularized operators $T_r$ with kernel $K_r(z, \xi)= 1/(1-r\overline \xi z)$ are uniformly bounded operators $L^2(\mu)\to L^2(v)$, so the operators $B_\gamma T_r$ are uniformly bounded $L^2(\mu)\to L^2$. 

On the other hand, the boundary values of the Cauchy transform $R$ (defined in \eqref{CauchyTrans}) exist a.e.~with respect to Lebesgue measure by the classical theory of Hardy spaces; it is easier than for the operators $\cV_\alpha$, since we do not need a.e.~convergence with respect to a singular measure here. 

In combination with the uniform bounds we can see the existence of weak operator topology limit
$$
T_{\pm}:=\text{w.o.t.-}\lim_{r\to 1^{\mp}} T_r.
$$
Note also that $T_\pm$ can be defined as a.e.~limits, $T_{\pm}f = \lim_{r\to 1^{\mp}} T_r f$.

\begin{theo}\label{t-repr-SNF}
Operator $\Phi_\gamma^*$ can be represented in the Sz.-Nagy--Foia\c{s} transcription as
\begin{align*}
(1-|\gamma|^2)^{1/2}\Phi_\gamma^* f
 & = 
\left( \begin{array}{c} 0 \\  (\overline\gamma - (\overline\gamma - 1) T_+ \1 )\Delta_\gamma \end{array}
\right) f 
+
\left( \begin{array}{c} (1 + \overline\gamma \theta_\gamma) / T_+\1 \\  (\overline\gamma - 1) \Delta_\gamma \end{array}
\right) T_+ f 
\\
\notag
 & 
=
\left( \begin{array}{c} 0 \\   \frac{1-\overline\gamma\theta_0}{|1-\overline\gamma\theta_0|}   T_+ \ID \cdot\Delta_0 \end{array}
\right) f 
+
\left( \begin{array}{c} \frac{1-|\gamma|^2}{1 - \overline\gamma \theta_0} \cdot\frac1{ T_+\ID} \\  (\overline\gamma - 1) \frac{(1-|\gamma|^2)^{1/2}}{|1-\overline\gamma\theta_0|} \Delta_0 \end{array}
\right) T_+ f 
\end{align*}
for $f\in L^2(\mu)$.
\end{theo}

As expected this formula reduces to the normalized Cauchy transform for $\gamma=0$ and inner functions $\te$. To see this, we notice that the second component collapses as $\Delta(z) = (1-|\te(z)|^2)^{1/2} = 0$ Lebesgue a.e.~on $\T$, and that $T_+f/T_+\ID$ is equal to the normalized Cauchy transform.

\begin{proof}[Idea of the proof]
For smooth functions $f$ we replace the term $1-\bar\xi z$ in the denominator of \eqref{f-reprB} by $1 - r \bar\xi z$ and take the limit as $r\to 1^-$. We obtain the same formula \eqref{f-reprB}. Since we also have weak convergence of the operators we have
\[
(T_+ f) (z) - f (z) (T_+ \ID)(z)  = \int_\T \frac{f(\xi) - f(z) }{1- \overline\xi z} d\mu(\xi), \qquad z\in\T  . 
\]

We extend the operator by continuity to all of $L^2(\mu)$ and derive
\begin{align}
\label{e-TPlus}
T_+\ID = 1/(1-\te_0)
\end{align}
from \eqref{CauchyTrans} through \eqref{theta_0-theta_gamma}. Technical computations then yield the desired formula.
\end{proof}

Interestingly, similar arguments show that
\begin{align}
\label{e-TMinus}
T_-\ID = -\bar\te_0/(1-\bar\te_0).
\end{align}

\subsection{Representing \texorpdfstring{$\Phi^*$}{Phi<sup>*} in the de Branges--Rovnyak transcription}
We translate the formula in the latter Theorem \ref{t-repr-SNF} from the Sz.-Nagy--Foia\c{s} transcription to  the de Branges--Rovnyak transcription, rather than starting from the universal representation formula in Theorem \ref{t-reprB}. This strategy seemed less cumbersome as we circumvent having to re-do much of the subtle work of regularizing singular integral operators. Also, we found it refreshing to understand the connection between the transcriptions.

By virtue of  the definition of the Sz.-Nagy--Foia\c{s} model space $\cK_\theta$, see \eqref{K_theta}, a function 
\[
g=\left(\begin{array}{c} g_1\\ g_2\end{array}\right) \in \left(\begin{array}{c} H^2 \\ \clos\Delta L^2 \end{array} \right)
\] 
is in $\cK_\theta$ if and only if 
\begin{align}
\label{deBrangesRepr-01}
g_- := \overline \theta g_1 + \Delta g_2 \in H^2_- := L^2 (\T) \ominus H^2 .
\end{align}
Note, that knowing $g_1$ and $g_-$ one can restore $g_2$ on $\T$:
\[
g_2\Delta = g_-  -   g_1\overline \theta. 
\]

The equality \eqref{deBrangesRepr-01} means that the pair $g_+ = g_1$ and $g_-$ belongs to the de Branges--Rovnyak space, see \eqref{deBrangesRepr}. 
It is also not hard to check that the norm of the pair $(g_1, g_-)$ in the Branges--Rovnyak space (i.e.~in the weighted space $L^2(W)$, 
$W=W_\theta^{[-1]}$, see Subsection~\ref{s:deBrangesRepr}) coincides with the norm of the pair $(g_1, g_2)$ in the Sz.-Nagy--Foia\c{s} 
space (i.e.~in non-weighted $L^2$). Indeed, we have
\begin{align*}
\left(\begin{array}{c} g_1 \\ g_- \end{array}\right) = 
\left(\begin{array}{cc} 1 & 0 \\ \overline\theta & \Delta \end{array}\right)
\left(\begin{array}{c} g_1 \\ g_2 \end{array}\right)\,.
\end{align*}
Let $B$ be a ``Borel support'' of $\Delta$, i.e.~the set where one of the representative from the equivalence class of $\Delta$ is different from $0$. 
A direct computation shows that for 
\[
W_\theta = \left(\begin{array}{cc} 1 & \theta \\ \overline\theta & 1 \end{array}\right)
\]
we have a.e.~on $\T$ 
\begin{align*}
\left(\begin{array}{cc} 1 & \theta \\ 0 & \Delta \end{array}\right)
W_\theta^{[-1]}\left(\begin{array}{cc} 1 & 0 \\ \overline\theta & \Delta \end{array}\right)
=
\left(\begin{array}{cc} 1 & 0 \\ 0 & \ID\ci B \end{array}\right) ,  
\end{align*}
which gives the desired equality of the norms. 
(Here $ \ID\ci B$ denotes the characteristic function of $B$.)

Note that functions in $H^2_-$ admit analytic continuation to the exterior of the unit disc, so a function in $\cK_\theta$ is determined by the boundary values of two functions $g_1$ and $g_-$ analytic in $\D$ and $\text{ext}( \overline\D)$ respectively. 

Since the first component of a function in the Branges--Rovnyak space is the same as in the Sz.-Nagy--Foia\c s space and by virtue of Theorem \ref{t-repr-SNF} we immediately know
\begin{align}
\label{g_+}
g_+(z) = g_1(z) 
= \frac{(1-|\gamma|^2)^{1/2}}{1-\overline{\gamma}\theta_0}\, \frac{T_+f}{T_+\ID} 
 \,,
\end{align}
where $T_+$ was defined in the paragraph prior to Theorem \ref{t-repr-SNF}.

The second component $g_-=g_-^\gamma$ is analytic on $\text{ext}(\overline\D)$. Therefore, we do need to return to the universal representation formula. After some reformulation we observe.

\begin{theo}[Theorem 5.5 of \cite{LT}]
\label{t:g_-}
Let $\mu$ be not the Lebesgue measure. 
Then the function $g_-=g^\gamma_-$ is given by 
\begin{align}
\label{g_-}
g^\gamma_- = (1-|\gamma|^2)^{-1/2} \left( \overline \theta_\gamma +\overline\gamma \right) \frac{T_-f}{T_-\ID} 
& = \frac{(1-|\gamma|^2)^{1/2} \overline\theta_0 }{1-\gamma \overline\theta_0}  
\cdot 
\frac{T_-f}{T_-\ID}  
 \,.
\end{align}
\end{theo}

\subsection{Formulas for \texorpdfstring{$\Phi_\gamma$}{Phi<sub>gamma}}
A representation of the Clark operator $\Phi_\gamma$ is given in terms of the components $g_+$ and $g_-$ of a vector in the de Branges--Rovnyak transcription. This formula is given piecewise. For a function $f\in L^2(\mu)$ we denote by $f\ti{a}$ and $f\ti{s}$ its ``absolutely continuous" and ``singular" parts, respectively. 
Formally, $f\ti s$ and $f\ti a$ can be defined as Radon--Nikodym derivatives $f\ti s = d(f\mu)\ti s/ d\mu\ti s$, $f\ti a = d(f\mu)\ti a/ d\mu\ti a$. 

Let $w$ denote the density of the absolutely continuous part of $d\mu$, i.e.~$w = d\mu/dx\in L^1$.

\begin{theo}
\label{t:ReprPhi}
Let $g = \left(\begin{array}{c} g_+\\ g_- \end{array}\right) \in \cK_{\theta_\gamma}$ (in the de Branges--Rovnyak transcription) and let $f\in L^2(\mu)$, $f= \Phi_\gamma g$. Then
\begin{enumerate}
\item the non-tangential boundary values of the function 
\[
z\mapsto\frac{1-\overline\gamma}{(1-|\gamma|^2)^{1/2}} g_+(z), \qquad z\in \D
\]
exist and coincide with $f\ti s$ $\mu\ti s$-a.e.~on $\T$. 
 
\item for the ``absolutely continuous'' part $f\ti a$ of $f$ 
\[
(1-|\gamma|^2)^{1/2} w f\ti a  = \frac{1-\overline\gamma\theta_0}{1-\theta_0} g_+ + \frac{1-\gamma\overline\theta_0}{1-\overline\theta_0} g_-
\]
a.e.~on $\T$.  
\end{enumerate}
\end{theo}

We provide the idea of the proof. First consider statement $(2)$.
By taking the limit as $r\to 1$ in $T_rf - T_{1/r}f$ we prove the Fatou type result \cite[Lemma 5.6]{LT}:
\[
T_+ f - T_- f = w f\quad\text{a.e.~on }\T
\]
(with respect to the Lebesgue measure) for all $f\in L^2(\mu)$. 
Together with \eqref{e-TMinus} and \eqref{e-TPlus} we can use the representations \eqref{g_+} and \eqref{g_-} for $g_+$ and $g_-$ to see the desired result for the absolutely continuous part.

Statement $(1)$ uses Poltoratskii's theorem \cite[Theorem 2.7]{NONTAN}.

\subsection{Clark operator for other \texorpdfstring{$\alpha$, $|\alpha|=1$}{alpha, |alpha|=1}}
Consider the Clark operator $\Phi_{\alpha, \gamma} :\cK_{\theta_\gamma} \to L^2(\mu_\alpha)$, where $\mu_\alpha$, $|\alpha|=1$ is the spectral measure corresponding to the cyclic vector $b$ of the unitary operator $U_\alpha$. Operator $\Phi_{\alpha, \gamma}$ is a unitary operator, which intertwines the model operator $\cM_{\te_\gamma}$ and the c.n.u.~contraction $(U_\gamma)_\alpha$ which is the operator $U_\gamma$ in the spectral representation of the operator $U_\alpha$.

We deduce everything from the results we already obtained. First, let us write 
the c.n.u.~contraction $U_\gamma$, $|\gamma|<1$ as a rank one perturbation of the unitary operator $U_\alpha$, $|\alpha|=1$:
\[
U_\gamma = U + (\gamma -1 ) b b_1^*   = U + (\alpha -1) bb_1^* + (\gamma - \alpha) b b_1^* 
 = U_\alpha + (\gamma/\alpha -1 ) b \wt b_1^*,  
\]
where $\wt b_1=\overline\alpha b_1$. 


From now we can just read off the results for $\alpha\in\T$ from the results we already proved (for $\alpha=1$). 

But to be consistent, we need the operators $\Phi_{\alpha, \gamma}^*$ to agree. First we need them to the same model spaces, so let us fix the model spaces to be the ones we got for the case $\alpha=1$. Second, we want them to be consistent with respect to the operators $\cV_\alpha$ from Section \ref{s:UnitRk1}:
%
\begin{align}
\label{Phi-V}
\Phi_{\alpha, \gamma}^* =\Phi_\gamma^* \cV_\alpha^*  \,.
\end{align}
Then an appropriately interpreted  ``universal'' representation formula (Theorem \ref{t-reprB}) gives us a 
formula for $\Phi_{\alpha, \gamma}^*$. 

Namely, in the spectral representation of $U_\alpha$ the c.n.u.~contraction $U_\gamma$ is given by
\begin{align}
\label{U_alpha_gamma_01}
M_\xi + (\gamma/\alpha - 1) b^\alpha (b_1^\alpha)^*, 
\end{align}
where $b^\alpha = \cV_\alpha b$, $b_1^\alpha = \cV_\alpha\wt b_1 =\overline \alpha \cV_\alpha b_1$, which yields $b^\alpha =\1$, $b_1^\alpha(\xi) \equiv \overline \xi$, $\xi\in \T$.
Notice that 
\[
c^{\alpha, \gamma} = \Phi_{\alpha,\gamma}^* b^\alpha = \Phi^*_\gamma \cV_\alpha^* b^\alpha = \Phi_\gamma^* b = c^\gamma, 
\]
and that
\[
c_1^{\alpha, \gamma} =\Phi_{\alpha,\gamma}^* b_1^\alpha = \Phi^*_\gamma \cV_\alpha^* b_1^\alpha 
=\overline \alpha\Phi^*_\gamma  b_1 = \overline \alpha c_1^\gamma. 
\]

Therefore, to get the formula for $\Phi_{\alpha, \gamma}^*$ with $ \Phi_{\alpha, \gamma}^* b^\alpha = c^\gamma$ (i.e.~such that $\Phi_{\alpha,\gamma}^* \1 = c^\gamma$) one just has to replace in \eqref{f-reprB} $\mu$ by $\mu_\alpha$, 
and $c_1^\gamma$ by $\overline\alpha c_1^\gamma$ ($c^\gamma$ remains the same). Note, that as long as $c^\gamma$ and $c_1^\gamma$ are computed, the parameter $\gamma$ does not appear in \eqref{f-reprB}. 


Now let us get the representations in the Sz.-Nagy--Foia\c{s} and de Branges--Rovnyak transcriptions.  One of the ways to get the formula for $\Phi_{\alpha,\gamma}^*$ would be to take the ``universal formula'' above and then repeat the proofs of Theorem \ref{t-repr-SNF} and of Theorem \ref{t:g_-}.

But the there is a simpler (in our opinion) way, that allows us to get the result with almost no computations: one just have to ``translate'' Theorems \ref{t-repr-SNF}, \ref{t:g_-} to the spectral representation of $U_\alpha$. 

In both these theorems the characteristic function and the parameter $\gamma$ are included explicitly, so we need to see how they change when we move to the spectral representation of $U_\alpha$. 

If we want to apply know formulas \eqref{e-CHARgamma}, they give us the characteristic function $\theta^{\alpha}_{\gamma/\alpha}$ of the  operator \eqref{U_alpha_gamma_01} with $b_1^\alpha$ and $b^\alpha$ taken for the basis vectors in the corresponding defect subspaces.

So, by replacing $\mu$ with $\mu_\alpha$ and $\gamma$ with $\gamma/\alpha$ in \eqref{e-CHARgamma} and \eqref{d-Delta} we get the characteristic function and the defect given by 
\[
\theta^\alpha_{\gamma/\alpha} = \overline \alpha\theta_\gamma\,, \qquad \text{and} \qquad \Delta_\gamma{\alpha} = \Delta_\gamma\,. 
\]

Substituting these functions to \eqref{e-AB}  and replacing $\gamma$ there by $\gamma/\alpha$ we get a representation formula for the adjoint of the Clark operator mapping $L^2(\mu_\alpha)\to\cK_{\bar\alpha\theta_\gamma} $ in Sz.-Nagy--Foia\c s transcription,
\begin{align}
\label{AltRepr-alpha}
 (1-|\gamma|^2)^{1/2}
 \wt \Phi_{\alpha,\gamma}^* f
 & = 
\left( \begin{array}{c} 0 \\  (\overline\gamma/\overline\alpha - (\overline\gamma/\overline\alpha - 1) T^\alpha_+ \1 )\Delta_\gamma \end{array}
\right) f 
+
\left( \begin{array}{c} (1 + \overline\gamma \theta_\gamma) / T^\alpha_+1 \\  (\overline\gamma/\overline\alpha - 1) \Delta_\gamma \end{array}
\right) T^\alpha_+ f 
\\
\notag
 & = 
\left( \begin{array}{c} 0 \\   \frac{1-\overline\gamma\theta_0}{|1-\overline\gamma\theta_0|}   T_+^\alpha \1 \cdot\Delta_0 \end{array}
\right) f 
+
\left( \begin{array}{c} \frac{1-|\gamma|^2}{1 - \overline\gamma \theta_0} \cdot\frac1{ T^\alpha_+\1} \\  (\overline\gamma /\overline\alpha - 1) \frac{(1-|\gamma|^2)^{1/2}}{|1-\overline\gamma\theta_0|} \Delta_0 \end{array}
\right) T_+^\alpha f \,,
\end{align}
where we let $T^\alpha_+f$ denote the non-tangential boundary values of $Rf\mu_\alpha(z)$, $z\in\D$.

But the above formula  is not yet the formula we are looking for! To get it we applied Theorem \ref{t-repr-SNF} with $\mu_\alpha$ instead of $\mu$ and $\theta^\alpha_{\gamma/\alpha} = \overline\alpha\theta_\gamma$ instead of $\theta_\gamma$. But that means that the result in the right hand side there
 belongs to $\cK_{\overline\alpha \theta}$. So the the above expression is an absolutely correct formula giving the representation of the operator $\Phi^*_{\alpha,\gamma}$ in the model space $\cK_{\overline\alpha\theta_\gamma}$; that is why we used $\wt \Phi^*_{\alpha,\gamma}$ and not $\Phi^*_{\alpha,\gamma}$ there.

To get the representation with the model space $\cK_{\theta_\gamma}$ we notice that the map
\[
\left( \begin{array}{c} g_1\\ g_2 \end{array} \right) 
\mapsto 
\left( \begin{array}{c} g_1\\ \overline\alpha g_2 \end{array} \right)
\]
is a unitary map from $\cK_{\overline\alpha\theta_\gamma}$ onto $\cK_{\theta_\gamma}$. Moreover, it maps the defect vector $c$ given by equation \eqref{c} for the space $\cK_{\overline\alpha\theta_\gamma}$ to the corresponding defect vector $c$ for the space $\cK_{\theta_\gamma}$. Therefore, to obtain the representation formula for $\Phi_{\alpha,\gamma}^*$ we need to multiply the bottom entries in \eqref{AltRepr-alpha} by $\overline\alpha$, which gives us 


\begin{theo}
Operator $\Phi_{\alpha, \gamma}^*$ can be represented in the Sz.-Nagy--Foia\c{s} transcription as
\begin{align*}
 (1-|\gamma|^2)^{1/2}
 \Phi_{\alpha,\gamma}^* f
 & = 
\left( \begin{array}{c} 0 \\  (\overline\gamma - (\overline\gamma - \overline\alpha) T^\alpha_+ \1 )\Delta_\gamma \end{array}
\right) f 
+
\left( \begin{array}{c} (1 + \overline\gamma \theta_\gamma) / T^\alpha_+1 \\  (\overline\gamma -\overline\alpha) \Delta_\gamma \end{array}
\right) T^\alpha_+ f 
\\
\notag
 & = 
\left( \begin{array}{c} 0 \\  \overline\alpha \frac{1-\overline\gamma\theta_0}{|1-\overline\gamma\theta_0|}   T_+^\alpha \1 \cdot\Delta_0 \end{array}
\right) f 
+
\left( \begin{array}{c} \frac{1-|\gamma|^2}{1 - \overline\gamma \theta_0} \cdot\frac1{ T^\alpha_+\1} \\  (\overline\gamma  - \overline\alpha ) \frac{(1-|\gamma|^2)^{1/2}}{|1-\overline\gamma\theta_0|} \Delta_0 \end{array}
\right) T_+^\alpha f \,. 
\end{align*}
\end{theo}

\section{Few remarks about Clark theory for the dissipative case}\label{s-SAmodel}
Consider a family of rank one perturbations similar to Section \ref{s-SA}, but with perturbation parameter $\alpha\in \C_+ := \{z\in \C:\im z>0\}$. In other words, in the spectral representation of $A$ (with respect to the cyclic vector $\f$ and spectral measure $\mu = \mu^\f$) we study the family of perturbations given by
\begin{align*}
A_\alpha = M_t + \alpha (\fdot, \ID)\ci{L^2(\mu)} \ID
\qquad\text{on }
L^2(\mu)
\text{ with }
\alpha\in \C_+.
\end{align*}
Recall that we consider the extended class of form bounded rank one perturbations. In the spectral representation this condition is equivalent to $$\int\ci\R\frac{d\mu(t)}{1+|t|}<\infty.$$

Without going into details on the definition of the perturbation in this case, we just say that one of the ways is to use the resolvent formula \eqref{singres}.

While there is no ``canonical'' model for the dissipative operator, a widely accepted way is  to construct the model for the Cayley transform $\wt T_\alpha=(A_\alpha-i\I)(A_\alpha + i\I)^{-1}$. 

So, let us compute $\wt T_\alpha$, introducing some notation along the way. 

Denote by $\wt U$ the Cayley transform of $A=A_0$, $\wt U= (A-i\I)(A+i\I)^{-1}$.  Using the resolvent formula \eqref{singres} and denoting 
\[
\wt b := \|(A+i\I)^{-1} \f\|^{-1} (A+i\I)^{-1} \f, \qquad
\wt b_1 := \|(A-i\I)^{-1} \f\|^{-1} (A-i\I)^{-1}
\]
we can write 
\[
T_\alpha = \wt U_\gamma=\wt U + (\gamma-1)\wt b (\wt b_1)^*, 
\]
where 
\begin{align}
\label{gamma(alpha)}
\gamma = \gamma(\alpha) = \frac{1+\alpha\overline Q}{1+\alpha  Q} , \qquad Q= ((A+i\I)^{-1} \f, \f) = \int_\R \frac{d\mu(s)}{s+i} . 
\end{align}
If we denote 
\[
F(z):= \int_\R \frac{d\mu(s)}{s-z},  
\]
we get that $Q=F(-i)=\overline{F(i)}$.

Note also that $\|\wt b\| =\|\wt b_1\|=1$ and $b_1= \wt U^* b$. It is obvious that $\gamma(\alpha)\in \T$ for $\alpha\in\R$.  Since $\im Q<0$, we conclude that  $\gamma(\alpha)\in\D$ for $\im \alpha>0$.  Thus $T_\alpha$ is a contractive rank one perturbation of the unitary operator $\wt U$. Under our assumptions about cyclicity of $\f$, one can easily see that $\wt b$ is a $*$-cyclic vector for $\wt U$, so $\wt U $ is unitarily equivalent to the multiplication $U=M_\xi$ by the independent variable $\xi$ in $L^2(\mu\ci \T)$, where  $\mu\ci \T$ is the spectral measure of $\wt U$ corresponding to the vector $\wt b$. 

Let us fix some notation: for $\gamma=\gamma(\alpha)$ given by \eqref{gamma(alpha)} we denote $\wt U_\gamma = \wt T_\alpha$, and by $U_\gamma=T_\alpha$ we denote the representation of the same operator in $L^2(\mu\ci\T)$. In other words, we use $T$ in conjunction with the parameter $\alpha\in\C_+$ and $U$ in conjunction with the parameter $\gamma=\gamma(\alpha)\in\D$; also we use $\widetilde{\,\,\,\,}$ for the operators in $L^2(\mu)$, and $T$ and $U$ act in $L^2(\mu\ci \T)$. 

The spectral measure $\mu\ci\T$ of $\wt U$ is easily computed. Namely, if $\omega$ denotes the standard conformal map from $\C_+$ to $\D$ (and from $\R$ to $\T$), 
\begin{align*}
\omega(z):= \frac{z-i}{z+i}, \qquad \omega^{-1}(\xi) = i\frac{1+\xi}{1-\xi}\,,
\end{align*}
then one can easily see that 
\begin{align*}
\mu\ci\T: = \wt \mu \circ\omega^{-1}, \qquad \text{where }  d\wt \mu(x) =\frac{1}{P}\cdot \frac{d\mu(x)}{1+x^2};  
\end{align*}
here by $\wt \mu\circ \omega^{-1}$ we mean that  $\wt\mu\circ \omega^{-1}(E) = \wt\mu(\omega^{-1}(E))$, $E\subset \T$, and 
$
P:= \int\ci\R\frac{d\mu(t)}{1+t^2}
$.

%
%
\subsection{What is the model for the dissipative case?}

As we mentioned above, it is customary for dissipative operator to consider for the model the model for its Cayley transform. Using formulas  \eqref{e-CHARgamma} with $\mu\ci \T$ for $\mu$, and the above description of $\mu\ci \T$  we can write a the characteristic function $\theta_\gamma$, $\gamma=\gamma(\alpha)$. 

However, since our original objects live on the real line (in $L^2(\mu)$), it is natural to consider the model also to be a space of functions on the real line. The standard unitary mapping $\Omega: L^2(\T)\to L^2(\R)$, 
\[
\Omega f(x) : = \frac{1}{\sqrt\pi (x+i)} f\circ \omega(x)  
\]
maps $H^2(\D)$ onto $H^2(\C_+)$ and so $H^2_-(\D)$ onto $H^2_-(\C_+)$. So if we use $\Omega^{-1}$ to transfer the model space $\cK_\theta$ to the space of functions on $\R$, the model space on the real line in Sz.-Nagy--Foia\c s and the de Branges--Rovnyak transcriptions will be defined exactly the same way as the model space on the circle. 

The multiplication $M_\omega$ by the function $\omega$ on $\R$ corresponds to the multiplication  by $\xi$ on $\T$. 

Note also that $(\pi/P)^{1/2}\Omega$ is a unitary operator $L^2(\mu\ci \T)\to L^2(\mu)$ and that the map $f\mapsto f\circ \omega$ maps $L^2(\mu\ci\T)\to L^2(\wt\mu)$ unitarily.

\subsubsection{Characteristic function in the half-plane}
Let us now compute the characteristic function for $A_\alpha$. For $\gamma=\gamma(\alpha)$ defined by \eqref{gamma(alpha)} let $\theta_\gamma$ be the characteristic function of $T_\alpha$ computed with respect to the vectors $\wt b_1$ and $\wt b$. Let
\[
\Theta_\alpha=\wt \theta_\gamma := \theta_\gamma\circ \omega
\]
be the transfer of $\theta_\gamma$ from the disc to the half-plane; note that we use capital $\Theta$ in conjunction with the parameter $\alpha\in\C_+$. 

Let us now transfer the Cauchy type integrals \eqref{CauchyTrans} to $\C_+$. For $w\in \C_+$ let $\lambda=\omega(w)$. Then 
\begin{lem}
We have
\begin{align}\label{e-TransR}
R \mu\ci{\T} (\la)
&=
 \int_\T \frac{d\mu\ci{\T}(\xi)}{1-\overline\xi\la}
&=
\frac{1}{2iP} 
\int_\R
\left[\frac1{x-w}
-
\frac{1}{x+i}
\right]
d\mu\ci{\R}(x) & =:\wt R \mu(w),\\
\label{e-TransR1}
R_1 \mu\ci{\T} (\la)
&=
 \int_\T \frac{\overline\xi\la d\mu\ci\T(\xi)}{1-\overline\xi\la}
 &=
\frac{1}{2iP} 
\int_\R
\left[\frac1{x-w}
-
\frac{1}{x-i}
\right]
d\mu\ci{\R}(x) &=: \wt R_1\mu (w),\\
\label{e-TransR2}
R_2 \mu\ci{\T} (\la)
&=
 \int_\T \frac{1+\overline\xi\la }{1-\overline\xi\la}d\mu\ci{\T}(\xi)
&=
\frac{1}{iP}
\int_\R
\left[\frac1{x-w}
-
\frac{1}{x^2+1}
\right]
d\mu\ci{\R}(x) &=:\wt R_2\mu(w).
\end{align}
\end{lem}

%
%

Using formulas \eqref{e-CHARgamma} for the disc we can write the characteristic functions as 
\begin{align}\label{e-CHARgamma-01}
\wt\theta_\gamma(\la) =  -\gamma + \frac{(1-|\gamma|^2) \wt R_1 \mu(\la)}{1 + (1-\overline\gamma) \wt R_1 \mu(\la)} = 
\frac{(1-\gamma) \wt R_2\mu(\la) -(1+\gamma)}{(1-\overline\gamma )\wt R_2\mu(\la)+(1+\overline\gamma)}\,, \qquad \la\in \C_+.
\end{align}
Note that the formulas for $\wt\theta_0$ ($\gamma=0$, equivalently $\alpha=-1/\overline Q = -1/F(i)$) are especially simple. 
And $\wt\te_0$ is related to $\wt\te_\gamma$ by a fractional transformation:
\begin{align*}
\wt\theta_\gamma = \frac{\wt\theta_0-\gamma}{1-\overline\gamma\wt\theta_0} \qquad
\textup{or equivalently } \qquad \wt\theta_0 = \frac{\wt\theta_\gamma + \gamma}{1 + \overline\gamma\wt\theta_\gamma} \,.
\end{align*}

\subsubsection{Model and defect vectors  in the half-plane}
Recall that the model operator $\cM_{\wt\theta_\gamma}$ is the compression of the multiplication operator $M_\omega$ by the function $\omega$, 
\[
 \cM_{\wt\theta_\gamma} f:=  P\ci{\cK_{\wt \theta_\gamma}} M_\omega f, \qquad f\in\cK_{\wt\theta_\gamma}. 
\]

Let us compute defect subspaces of $\cM_{\wt\theta_\gamma}$. 

Considering vectors $c=c^\gamma$ and $c_1=c_1^\gamma$ defined by \eqref{c}, \eqref{c_1} with $\mu\ci \T$ instead of $\mu$, define $\wt c:= c\circ \omega$,  $\wt c_1:= c_1\circ \omega$, $\wt\Delta :=\Delta\circ\omega =1-|\wt\theta_0|^2$. Computing we get in the Sz.-Nagy--Foia\c s transcription
\begin{align*}
\wt c(z) &:= \left( 1- |\wt\theta_\gamma(i)|^2 \right)^{-1/2} 
\left( \begin{array}{c} 1-\overline{\wt\theta_\gamma(i)}\wt\theta_\gamma (z)\\ -\overline{\wt\theta_\gamma(i)}\wt\Delta (z) \end{array} \right), 
\\
\wt c_1(z) & := \left( 1- |\wt\theta_\gamma(i)|^2 \right)^{-1/2} 
\left( \begin{array}{c} \omega(z)^{-1} \left(\wt\theta_\gamma(z)- \wt\theta_\gamma(i)\right) \\ \omega(z)^{-1} \wt\Delta (z)\end{array} \right).
\end{align*}
Then the defect subspaces $\fD\ci{\cM_{\wt\theta_\gamma}^*}$ and $\fD\ci{\cM_{\wt\theta_\gamma}}$ of $\cM_{\wt\theta_\gamma}$ are spanned by the vectors
\[
\frac{\wt c(z)}{\sqrt\pi(z+i)}, \qquad \frac{\wt c_1(z)}{\sqrt\pi(z+i)}
\]
and these vectors agree. 

\subsection{Representations of the adjoint Clark operator in the half-plane} 
Using these formulas we can transfer the universal representation formula given by Theorem \ref{t-reprB} from the unit circle $\T$ to the real line $\R$. For a function $f$ on the real line $\R$ define $\wt f$ by \[
\wt f(x):= (x+i)\cdot f(x). 
\] 
and let $f\ci\T := \wt f\circ \omega^{-1}$. Then we can easily transfer Theorem \ref{t-reprB} from the disc $\D$ to the  half-plane $\C_+$.

To simplify the notation let us assume that the measure $\mu$ is Poisson normalized, i.e.~that 
\[
P:= \int_\R \frac{d\mu(x)}{x^2+1} =1. 
\]
Formulas for the general case $P\ne 1$ can be then obtained if one notice that the map $f\mapsto P^{1/2}f$ is a unitary map  $L^2(\mu) \to  L^2(\mu/P)$. 
\subsubsection{A universal representation formula}
\begin{theo}[A ``universal" representation formula for dissipative perturbations]
\label{t-reprB-01}
Let the measure $\mu$ be Poisson normalized ($P=1$). 
Let $\wt\theta_\gamma$  be the characteristic function  of $T_\alpha=\wt U_\gamma$,  $|\gamma|<1$,   computed with respect to the vectors $\wt b_1$ and $\wt b$  (note that $\wt\theta_\gamma$ is given by \eqref{e-CHARgamma-01}). Let $\cK_{\wt\theta_\gamma}$ and $\cM_\gamma =\cM_{\wt\theta_\gamma}$ be the model subspace and the model operator respectively. 
Let $\wt\Phi^*_\gamma: L^2(\mu) \to \cK_{\wt\theta_\gamma}$ be the unitary operator satisfying 
\[
\wt\Phi_\gamma^* \wt U_\gamma = \cM_{\wt\theta_\gamma} \wt\Phi_\gamma^* , 
\]
and such that $\wt\Phi_\gamma^* \wt b (z)= \wt c^\gamma (z)/(\sqrt\pi(z+i))$, $\wt\Phi_\gamma^* \wt b_1 (z) = \wt c_1^\gamma(z)/(\sqrt\pi(z+i))$. 

Then for all compactly supported $f\in C^1(\R)$
\begin{align*}
\sqrt\pi (z+i)
\wt\Phi_\gamma^* f(z) &= 
\\
\notag
=\wt A_\gamma(z) \wt f(z) &+
\wt B_\gamma(z)\int \left(\wt f(s)- \wt f(z)  \right)\frac{1}{2i} \left[   \frac{1}{s-z}-\frac{1}{s+i}\right]\, d\mu(s)
\end{align*}
where $\wt A_\gamma(z) = \wt c^\gamma(z)$, $\wt B_\gamma(z) = \wt c^\gamma(z) - \omega(z) \wt c_1^\gamma(z)$.
\end{theo}

\subsubsection{A representation formula in the Sz.-Nagy--\texorpdfstring{Foia\c s}{Foias} transcription}
For a measure $\mu$ on the real line define $T_+f$ to be the non-tangential boundary values of the function 
\[
\wt R f\mu(z) = \frac1{2iP}\int_\R f(s)  \left[   \frac{1}{s-z}-\frac{1}{s+i}\right]\, d\mu(s), \qquad \im z>0, 
\] 
and let $T^1_+ f$ be the non-tangential boundary values of
\[
\wt R^1 f\mu(z) := \frac{1}{2iP}\int_\R \frac{f(s) d\mu(s)}{s-z} , \qquad \im z>0;
\]
the non-tangential boundary values exist a.e.~with respect to the Lebesgue measure by classical result about boundary values of the functions in the Hardy spaces $H^p$. 

\begin{theo}\label{t-repr-SNF-01} Let $\mu$ be Poisson normalized, $P=1$. 
The operator $\wt\Phi_\gamma^*$ can be represented in the Sz.-Nagy--Foia\c{s} transcription as
\begin{align*}
\sqrt\pi(1-|\gamma|^2)^{1/2}\wt\Phi_\gamma^* f
 & = 
\left( \begin{array}{c} 0 \\  (\overline\gamma - (\overline\gamma - 1) T_+ \1 )\wt\Delta_\gamma \end{array}
\right) f 
+
\left( \begin{array}{c} (1 + \overline\gamma \wt\theta_\gamma) / T_+\1 \\  (\overline\gamma - 1) \wt\Delta_\gamma \end{array}
\right) T_+^1 f 
\\
\notag
 & 
=
\left( \begin{array}{c} 0 \\   \frac{1-\overline\gamma\wt\theta_0}{|1-\overline\gamma\wt\theta_0|}   T_+ \ID \cdot\wt\Delta_0 \end{array}
\right) f 
+
\left( \begin{array}{c} \frac{1-|\gamma|^2}{1 - \overline\gamma \wt\theta_0} \cdot\frac1{ T_+\ID} \\  (\overline\gamma - 1) \frac{(1-|\gamma|^2)^{1/2}}{|1-\overline\gamma\wt\theta_0|} \wt\Delta_0 \end{array}
\right) T_+^1 f 
\end{align*}
for $f\in L^2(\mu)$.
\end{theo}

The same recipe as above gives the representation in the de Branges--Rovnyak transcription, and a formula for the Clark operator $\Phi_\gamma$.

\providecommand{\bysame}{\leavevmode\hbox to3em{\hrulefill}\thinspace}
\providecommand{\MR}{\relax\ifhmode\unskip\space\fi MR }
\providecommand{\MRhref}[2]{%
  \href{http://www.ams.org/mathscinet-getitem?mr=#1}{#2}
}
\providecommand{\href}[2]{#2}

\end{document}